\documentclass[11pt]{article}
\usepackage[margin=0.9in]{geometry}
\usepackage{bm}
\usepackage{graphicx}
\usepackage{amssymb,amsmath,amsthm,mathrsfs}
\usepackage{titlesec}
\usepackage{epstopdf}
\usepackage{algorithm}
\usepackage{amsfonts,delarray}
\usepackage{algorithm,algcompatible,amsmath}
\usepackage{rotating,color}
\usepackage{subfigure}
\usepackage{natbib,multirow}
\usepackage{titlesec,textcase}
\usepackage{longtable}
\usepackage{bbm}
\usepackage{rotating}
\usepackage{authblk}
\usepackage{times,url}
\usepackage{dblfloatfix}
\usepackage[flushleft]{threeparttable}
\usepackage{arydshln}
\usepackage{graphicx}
\usepackage[colorlinks=true, allcolors=blue]{hyperref}

\newtheorem{lemma}{Lemma}

\newtheorem{definition}{Definition}

\newtheorem{theorem}{Theorem}
\newtheorem{corollary}[theorem]{Corollary}

\DeclareMathOperator*{\argmin}{arg\,min}
\DeclareMathOperator*{\argmax}{arg\,max}

\definecolor{rp}{RGB}{83,54,106}

\def\boxit#1{\vbox{\hrule\hbox{\vrule\kern6pt\vbox{\kern6pt#1\kern6pt}\kern6pt\vrule}\hrule}}

\begin{document}

\title{Low noise sensitivity analysis of $\ell_q$-minimization in oversampled systems
\thanks{}
}
\author{Haolei Weng and Arian Maleki}

\date{}
\maketitle

\begin{abstract}

The class of $\ell_q$-regularized least squares (LQLS) are considered for estimating $\beta \in \mathbb{R}^p$ from its $n$ noisy linear observations $y=X\beta+ w$. The performance of these schemes are studied under the high-dimensional asymptotic setting in which the dimension of the signal grows linearly with the number of measurements. In this asymptotic setting, phase transition diagrams (PT) are often used for comparing the performance of different estimators. PT specifies the minimum number of observations required by a certain estimator to recover a structured signal, e.g. a sparse one, from its noiseless linear observations. Although phase transition analysis is shown to provide useful information for compressed sensing, the fact that it ignores the measurement noise not only limits its applicability in many application areas, but also may lead to misunderstandings. For instance, consider a linear regression problem in which $n>p$ and the signal is not exactly sparse. If the measurement noise is ignored in such systems, regularization techniques, such as LQLS, seem to be irrelevant since even the ordinary least squares (OLS) returns the exact solution. However, it is well-known that if $n$ is not much larger than $p$ then the regularization techniques improve the performance of OLS. 

In response to this limitation of PT analysis, we consider the {\em low-noise sensitivity analysis}. We show that this analysis framework (i) reveals the advantage of LQLS over OLS, (ii) captures the difference between different LQLS estimators even when $n>p$, and (iii) provides a fair comparison among different estimators in high signal-to-noise ratios. As an application of this framework, we will show that under mild conditions LASSO outperforms other LQLS even when the signal is dense. Finally, by a simple transformation we connect our low-noise sensitivity framework to the classical asymptotic regime in which $n/p \rightarrow \infty$ and characterize how and when regularization techniques offer improvements over ordinary least squares, and which regularizer gives the most improvement when the sample size is large.

\end{abstract}

\section{Introduction}\label{sec:intro}

\subsection{Problem Statement} \label{ps:first}

In modern data analysis, one of the fundamental models that has been extensively studied is the high-dimensional linear model $y=X\beta+w$ with $y\in \mathbb{R}^n, X \in \mathbb{R}^{n\times p}$, where the length of the signal $p$ is at the same order or in some cases larger than the number of observations $n$. Since the ordinary least squares (OLS) estimate is not accurate in the high-dimensional regime, researchers have proposed a wide range of regularization techniques and recovery algorithms to go beyond OLS. The existence of a variety of algorithms and regularizers has in turn called for platforms that can provide fair comparisons among them. One of the most popular platforms is the phase transitions analysis (PT). Intuitively speaking, phase transition diagram measures the minimum number of observations an algorithm or an estimator requires to recover $\beta$. To understand the limitations of PT we study the performance of the following $\ell_q$-regularized least squares (LQLS), a.k.a. bridge estimators \cite{frank1993statistical},
\begin{eqnarray}\label{lqls:def}
\hat{\beta}(\lambda, q) \in \argmin_{\beta}\frac{1}{2}\|y-X\beta\|_2^2+\lambda \|\beta\|_q^q,  \quad 1\leq q\leq 2,
\end{eqnarray}
where $\|\cdot\|_q$ is the usual $\ell_q$ norm and $\lambda \geq 0$ is a tuning parameter. This family covers LASSO \citep{tibshirani1996regression} and Ridge \citep{hoerl1970ridge}, two well known estimates in high-dimensional statistics  and compressed sensing. 

Phase transition analysis studies the asymptotic mean square error (AMSE) $\|\hat{\beta}(\lambda, q)- \beta \|_2^2/p$ under the asymptotic setting $p \rightarrow \infty$ and $n/p \rightarrow \delta$. Then, it considers $w=0$ and calculates the smallest $\delta$ for which $\inf_\lambda \lim_{p \rightarrow \infty} \|\hat{\beta}(\lambda, q)- \beta \|_2^2/p =0$. In this paper, we consider situations in which $\beta$ is not exactly sparse. As is intuitively expected and will be discussed later in the paper, the phase transition analysis implies that for every $1 \leq q \leq 2$, if $\delta>1$, then $\inf_\lambda \lim_{p \rightarrow \infty} \|\hat{\beta}(\lambda, q)- \beta \|_2^2/p =0$ and if $\delta <1$, then $\inf_\lambda \lim_{p \rightarrow \infty} \|\hat{\beta}(\lambda, q)- \beta \|_2^2/p \neq 0$. This simple application of PT reveals some of the limitations of the phase transition analysis:
\begin{enumerate}
\item Phase transition analysis is concerned with $w=0$, and when $\beta$ is not sparse, LQLS with different values of $q$ have the same phase transition at $\delta=1$. Hence, it is not clear whether regularization can improve the performance of  ordinary least squares (OLS) and if it does, which regularizer is the best. We expect the choice of regularizer to matter when we add some noise to the measurements.

\item Phase transition diagram is not sensitive to the magnitudes of the elements of $\beta$. Again, intuitively speaking, this seems to have a major impact on the performance of different estimators when the noise is present in the system.
\end{enumerate}

This paper aims to present a generalization of the phase transition, called low-noise sensitivity analysis, in which it is assumed that a small amount of noise is present in the measurements. Then it calculates AMSE of each estimate. This framework has the following two main advantages over the phase transition analysis:
\begin{enumerate}
\item It reveals certain phenomena that are important in applications and are not captured by PT analysis. For instance, one immediately sees the impact of the regularizer and the magnitudes of the elements of $\beta$ on the AMSE. Furthermore, these relations are expressed explicitly and can be interpreted easily.

\item It provides a bridge between the phase transition analysis proposed in compressed sensing, and the classical large sample-size asymptotic ($n/p \rightarrow \infty$). We will discuss some of the implications of this connection for the classical asymptotics in Section \ref{largesample:imp}. 

\end{enumerate}
As a consequence, low noise sensitivity analysis enables us to present a fair comparison among different LQLS, and reveal different factors that affect their performance.

Here is the organization of the rest of this paper: In Section \ref{sec:relatedwork}, we discuss some related works and highlight the differences from ours. In section \ref{sec:def}, we formally introduce the asymptotic framework adopted in our analyses. In Section \ref{sec:our}, we present and discuss our main contributions in details.  In Section \ref{sec:proof}, we prove all the main results.

\subsection{Related work}\label{sec:relatedwork}

The phase transition analysis for compressed sensing evolved in a series of papers by Donoho and Tanner \cite{donoho2004most, donoho2004most2, donoho2006high, donoho2005sparse, donoho2005neighborliness}. Donoho and Tanner characterized the phase transition curve for LASSO and some of its variants.
Inspired by Donoho and Tanner's breakthrough, many researchers explored the performance of different algorithms under the asymptotic settings $n/p \rightarrow \delta$ and $p \rightarrow \infty$ \cite{reeves2013minimax, stojnic2009various, stojnic2009block, stojnic2013under, amelunxen2013living, thrampoulidis2016precise, el2013robust,  karoui2013asymptotic, donoho2013high, donoho2013phase, donoho2015variance, bradic2015robustness, donoho2011noise, DoMaMoNSPT, foygel2014corrupted, zheng2015does,  rangan2009asymptotic, krzakala2012statistical, BaMo10, BaMo11, reeves2008sampling, reeves2016replica}. In this paper, we use the message passing analysis that was developed in a series of papers \cite{DoMaMoNSPT, donoho2009message, maleki2010approximate, BaMo10, BaMo11, maleki2013asymptotic} to characterize the asymptotic mean square errors of LQLS. As a result of this calculation we obtain some equations whose solution can specify AMSE of an estimate. Unfortunately, the complexity of these equations does not allow us to interpret the results and obtain useful information. Hence, we develop a machinery to simplify the AMSE formulas and turn them to explicit and informative quantities. 

Note that the AMSE formulas we derive in Theorem \ref{thm:eqpseudolip} can be calculated by the framework developed in \cite{thrampoulidis2016precise} too. Furthermore, the phase transition formulas we derive in Theorem \ref{thm:phasetransition} can be derived from the framework of \cite{amelunxen2013living} as well. The low noise sensitivity analysis that we use in this paper has been also used in \cite{weng2016overcoming}. But the analysis of \cite{weng2016overcoming} is only concerned with sparse signals. This paper avoids the sparsity assumption. Perhaps surprisingly, the proof techniques developed in \cite{weng2016overcoming} for the sparse signals come short of characterizing higher order terms in dense cases. In response to that limitation, we propose a delicate chaining argument in this paper that offers a much more accurate characterization of the higher order terms.
The study of dense signals provides much more complete understanding of the the bridge estimators. In particular, under the sparse settings, \cite{weng2016overcoming} reveals the monotonicity of LQLS's performance: LASSO is optimal, and the closer $q$ is to 1, the better LQLS performs. However, in this paper we show that the comparison and optimality characterization of LQLS becomes much more subtle for non-sparse signals. A general and detailed treatment for different types of dense signals are given in Section \ref{sec:our}.
Finally, we should also emphasize that the machinery required for the low-noise sensitivity analysis of dense signals is different from (and much more complicated) the one developed for sparse signals.

Our paper performs the  asymptotic analysis of LQLS. Many researchers have used non-asymptotic frameworks for this purpose. Among all the LQLS, LASSO is the best studied one. The past decade has witnessed dramatic progress towards understanding the performance of LASSO in the tasks of parameter estimation \citep{donoho2006compressed, donoho2005neighborliness, candes2006robust, donoho2009message, bickel2009simultaneous, raskutti2011minimax}, variable selection \citep{zhao2006model, wainwright2009sharp, meinshausen2006high, reeves2008sampling}, and prediction \citep{greenshtein2004persistence}. We refer the reader to \cite{buhlmann2011statistics, eldar2012compressed} for a complete list of references.  The ones most related to our work are \cite{donoho2006compressed, candes2006near, candes2011probabilistic}. In the first two papers, the authors considered non-sparse $\beta$ with the constraint that $\|\beta\|_q \leq R$ or the $i$th largest component $|\beta|_{(i)}$ decays as $i^{-\alpha} (\alpha>0)$. The papers derived optimal (up to logarithmic factor) upper bounds on the mean square error of LASSO. However, in this paper we characterize the performance of LASSO for a generic $\beta$ and derive conditions under which LASSO outperforms other bridge estimators. Also, we should emphasize that thanks to our asymptotic settings, unlike these two papers we are able to derive exact expressions of AMSE with sharp constants. Finally, \cite{candes2011probabilistic} studied a fixed signal $\beta$ and obtained an oracle inequality for $\|\hat{\beta}(1,\lambda)-\beta\|_2$, with the tuning $\lambda$ chosen as an explicit function of $p$. While their results are more general than ours, the bounds suffer from loose constants and are not sufficient to provide sharp comparison of LASSO with other LQLS. Moreover, the tuning parameter $\lambda$ in our case is set to the optimal one that minimizes the AMSE for every LQLS, which further paves our way for accurate comparison between different LQLS.   

Finally, the performance of LQLS with $q\geq 0$ under classical asymptotic setting where $p$ is fixed and $n\rightarrow \infty$ is studied in \cite{knight2000asymptotics}. The author obtained the $\sqrt{n}$ convergence of LQLS estimates and derived the asymptotic distributions. His results can be used to calculate the AMSE for LQLS with optimal tuning and show that they are all equal for $q\in [1,2]$. However, we demonstrate in Section \ref{largesample:imp} that by a second-order analysis, a more accurate comparison between the performances of different LQLS is possible. In particular, LASSO will be shown to outperform others for certain type of dense signals. We should also mention that the idea of obtaining higher order terms for large sample scenarios was first introduced in \cite{wang2017bridge}. However, all the results in \cite{wang2017bridge} are only concerned with the sparse signals. In this paper, we consider dense signals. We discussed the main challenges of switching from sparse signals to dense signals earlier in this section.


\section{The asymptotic framework}\label{sec:def}
The main goal of this section is to formally introduce the asymptotic framework under which we study LQLS. In the current section only, we write vectors and matrices as $X(p), \beta(p), y(p)$, and $w(p)$ to emphasize their dependence on the dimension of $\beta \in \mathbb{R}^p$.  Similarly, we may use $\hat{\beta}(\lambda,q,p)$ as a substitute for $\hat{\beta}(\lambda, q)$. We first define a specific type of a sequence known as a converging sequence. Our definition is borrowed from other papers \cite{DoMaMoNSPT, BaMo10, BaMo11} with some minor modifications. Recall that we have the linear model $y(p)=X(p)\beta(p)+w(p)$.

\begin{definition}\label{def:convseq}
A sequence of instances $\{X(p), \beta(p), w(p)\}$ is called a converging sequence if the following conditions hold:
\begin{itemize}
\item[-] $n/p\rightarrow \delta \in (0,\infty)$, as $n\rightarrow \infty$.
\item[-] The empirical distribution of $\beta(p) \in \mathbb{R}^p$ converges weakly to a probability measure $p_{\beta}$ with bounded second moment. Further, $\frac{1}{p} \|\beta(p)\|_2^2$ converges to the second moment of $p_{\beta}$. 
\item[-] The empirical distribution of $w(p) \in \mathbb{R}^n$ converges weakly to a zero mean distribution with variance $\sigma_w^2$. Furthermore, $\frac{1}{n} \|w(p)\|_2^2 \rightarrow \sigma_w^2$. 

\item[-] The elements of $X(p)$ are iid with distribution $N(0, 1/n)$. 
\end{itemize}
\end{definition}

For each of the problem instances in the converging sequence, we solve the LQLS problem \eqref{lqls:def} and obtain $\hat{\beta}( \lambda,q,p)$ as the estimate. The goal is to evaluate the accuracy of this estimate. Below we define the asymptotic mean square error as the asymptotic measures of performance. 

\begin{definition}\label{def:observablelasso}
Let $\hat{\beta}(\lambda,q,p)$ be the sequence of solutions of LQLS for the converging sequence of instances $\{\beta(p), X(p), w(p)\}$. The asymptotic mean square error of $\hat{\beta}(\lambda,q,p)$ is defined as the following almost sure limit
\begin{equation} \label{eq:def:AMSE}
{\rm AMSE}(\lambda,q,\sigma_w) \triangleq \lim_{p \rightarrow \infty} \frac{1}{p} \sum_{i=1}^p (\hat{\beta}_i(\lambda,q,p)-\beta_{i}(p))^2.
\end{equation}
\end{definition}
Note that AMSE depends on other factors like $p_{\beta}$. We have suppressed the notations for simplicity. The performance of LQLS as defined above is affected by the tuning parameter $\lambda$. In this paper, we consider the value of $\lambda$ that gives the minimum AMSE. Let $\lambda_{*,q}$ denote the value of $\lambda$ that minimizes AMSE given in \eqref{eq:def:AMSE}:\footnote{It turns out the optimal value of $\lambda$ can be estimated accurately under the asymptotic settings discussed in this paper. See \cite{mousavi2015consistent} for more information.}
\begin{equation}\label{eq:defoptlambda}
\lambda_{*,q}=\argmin_{\lambda\in [0,\infty)} {\rm AMSE}(\lambda,q,\sigma_w)
\end{equation}
 Then LQLS is solved with this specific value of $\lambda$, i.e., 
\begin{eqnarray}\label{optimal:lqls}
\hat{\beta} (\lambda_{*,q}, q ) \in \arg\min_{\beta} \frac{1}{2} \|y-X\beta\|_2^2+ \lambda_{*,q} \|\beta\|_q^q.
\end{eqnarray}
This is the best performance that LQLS with each value of $q$ can achieve in terms of AMSE. We can use Corollary 1 in \cite{weng2016overcoming} to obtain the precise formula of this optimal AMSE. 

\begin{theorem}\label{thm:eqpseudolip} 
For a given $q\in [1,2]$. Consider a converging sequence $\{\beta(p), X(p), w(p)\}$. Suppose that $\hat{\beta}(\lambda_{*,q},q)$ is the solution of LQLS with optimal tuning $\lambda_{*,q}$ defined in \eqref{optimal:lqls} and ${\rm AMSE}(\lambda_{*,q},q,\sigma_w)$ is its AMSE defined in \eqref{eq:def:AMSE}. Then 
\begin{equation}\label{eq:lassoobs}
{\rm AMSE}(\lambda_{*,q},q,\sigma_w) = \min_{\chi \geq 0}\mathbb{E}_{B,Z} (\eta_q(B+\bar{\sigma}Z; \chi)- B)^2,
\end{equation}
where $B$ and $Z$ are two independent random variables with distributions $p_\beta$ and $N(0, 1)$, respectively; $\eta_q$ is the proximal operator for the function $\| \cdot \|_q^q$; \footnote{The proximal operator of $\| \cdot \|_q^q$ is defined as $\eta_q (u; \chi) \triangleq \arg \min_z \frac{1}{2} (u-z)^2 + \chi |z|^q$. For further information on these functions, please refer to Lemma \ref{prox:prop} in Section \ref{begin:proof}.} and $\bar{\sigma}$ is the unique solution of the following equation:
\begin{eqnarray} \label{eq:fixedpoint}
\bar{\sigma}^2 &=& \sigma_{\omega}^2+\frac{1}{\delta} \min_{\chi \geq 0} \mathbb{E}_{B, Z} [(\eta_q(B +\bar{\sigma} Z; \chi) -B)^2].
\end{eqnarray}
\end{theorem}

Theorem \ref{thm:eqpseudolip} provides the first step in our analysis of LQLS. We first calculate $\bar{\sigma}$ from \eqref{eq:fixedpoint}. Then incorporating the solution $\bar{\sigma}$ into \eqref{eq:lassoobs} gives us the result of ${\rm AMSE}(\lambda_{*,q},q,\sigma_w)$. Given the distribution of $B$, the variance of the error $\sigma_w^2$, the number of observations (normalized by the number of predictors) $\delta$, and $q\in [1,2]$, it is straightforward to write a computer program to (numerically) find the solution of \eqref{eq:fixedpoint} and calculate the value of ${\rm AMSE}(\lambda_{*,q},q,\sigma_w)$.  However, it is needless to say that this approach does not shed much light on the performance of different LQLS estimates, since there are many factors involved in the computation and each affects the result in a non-trivial fashion. In this paper, we perform an analytical study on the solution of \eqref{eq:fixedpoint} and obtain an explicit characterization of AMSE in the high signal-to-noise ratio regime. The expressions we derived will offer us an accurate view of LQLS and quantify the impact of the distribution of $\beta$ on the performance of different LQLS. 

As is clear from the definition of a converging sequence in \ref{def:convseq} and Theorem \ref{thm:eqpseudolip}, the main property of $\beta$ that affects AMSE is the probability measure $p_{\beta}$. In the rest of the paper, we will assume $p_{\beta}$ does not have any point mass at zero. We use the notation $B$ to denote a one dimensional random variable distributed according to $p_{\beta}$. We present our findings for the high signal-to-noise ratio regime where the noise level $\sigma_w$ is either zero or small in Sections \ref{main:zero} and \ref{main:low}, respectively. We then discuss the implications of our analysis framework for classical asymptotics in Section \ref{largesample:imp}.

\section{Our main contributions} \label{sec:our}

\subsection{Phase transition } \label{main:zero}

Suppose that there is no noise in the linear model, i.e., $\sigma_w=0$. Our first goal in the phase transition analysis is to find the minimum value of $\delta$ for which ${\rm AMSE}(\lambda_{*,q},q, 0)=0$. Our next theorem characterizes the phase transition.  

\begin{theorem}\label{thm:phasetransition}
Let $q\in [1,2]$. If $\mathbb{E}|B|^2<\infty$, then we have
\begin{eqnarray*}
{\rm AMSE}(\lambda_{*,q},q,0)=
\begin{cases}
>0 & \mbox{~if~} \delta <1,  \\
=0 & \mbox{~if~}\delta >1.
\end{cases}
\end{eqnarray*}
\end{theorem}
The result can also be derived from several different frameworks including  the statistical dimension framework in \cite{amelunxen2013living}. But we derive it as a simple byproduct of our results in Section \ref{main:low}. So we do not discuss its proof here. This result is not surprising. Since, none of the coefficients is zero, the exact recovery is impossible if $n<p$.  Also, note that when $\delta>1$ even the ordinary least squares is capable of recovering $\beta$. Hence, the result of phase transition analysis does not provide any additional information on the performance of different regularizers. It is not even capable of showing the advantage of regularization techniques over the standard least squares algorithm. This is due to the fact that the result of Theorem \ref{thm:phasetransition} holds only in the noiseless case. Intuitively speaking, in the practical settings where the existence of the measurement noise  is inevitable, we expect different LQLS to behave differently. For instance, even though the signal under study is not sparse, when $p_\beta$ has a large mass around zero (it is approximately sparse), we expect the sparsity promoting LASSO to offer better performance than the other LQLS. However, the distribution $p_\beta$ does not have any effect on the phase transition diagram. Motivated by these concerns, in the next section, we investigate the performance of LQLS in the noisy setting, and study their noise sensitivity when the noise level $\sigma_w$ is small. The new analysis will offer more informative answers.

\subsection{Second-order noise sensitivity analysis of {\rm AMSE}} \label{main:low}

As an immediate generalization of the phase transition analysis, we can study the performance of different estimators in the presence of  a small amount of noise. More formally,  we derive the asymptotic expansion of ${\rm AMSE}(\lambda_{*,q},q,\sigma_w)$ for every $q\in [1,2]$, when $\sigma_w \rightarrow 0$. As will be discussed later, this generalization of phase transitions  presents  a more delicate analysis of LQLS. We start with the study of AMSE for the ordinary least squares (OLS). The result of OLS will be later used for comparison purposes.

\begin{lemma}\label{ols:property}

Consider the region $\delta>1$. For the OLS estimate $\hat{\beta}(0,q)$, we have
\[
{\rm AMSE}(0,q,\sigma_w)=\frac{\sigma^2_w}{1-1/\delta}.
\] 
\end{lemma}

We prove the above lemma in Section \ref{ols:proof}. Note that the proof we presented there, has not used the independence of the noise elements that is often assumed in the analysis of OLS. 
Now we can discuss LQLS with the optimal choice of $\lambda$ defined in \eqref{eq:defoptlambda}. It turns out that the distribution $p_\beta$ impacts ${\rm AMSE}(\lambda_{*,q},q,\sigma_w)$ in a subtle way. For analysis purposes we first study the signals whose elements are bounded away from zero in Theorem \ref{thm:densebbiggerthanmu} and then study other distributions in Theorem \ref{thm:densegeneralcase}.

\begin{theorem}\label{thm:densebbiggerthanmu}
Consider the region $\delta>1$. Suppose that $\mathbb{P}(|B|>\mu)=1$ with $\mu$ being a positive constant and $\mathbb{E}|B|^2<\infty$. Then, for $q\in (1,2]$, as $\sigma_w\rightarrow 0$
\begin{eqnarray*}
{\rm AMSE}(\lambda_{*,q},q,\sigma_w)=\frac{\sigma_w^2}{1-1/\delta}-\frac{\delta^{3}(q-1)^2(\mathbb{E} |B|^{q-2})^2}{(\delta-1)^3  \mathbb{E} |B|^{2q-2}} \sigma_w^4+o(\sigma_w^4),
\end{eqnarray*}
and for $q=1$, as $\sigma_w\rightarrow 0$
\begin{eqnarray*}
{\rm AMSE}(\lambda_{*,q},q,\sigma_w)=\frac{\sigma_w^2}{1-1/\delta}-|o(e^{-\frac{\tilde{\mu}^2(\delta-1)}{\delta\sigma_w^2}})|,
\end{eqnarray*}
where $\tilde{\mu}$ is any positive number smaller than $\mu$.
\end{theorem}

The proof can be found in Section \ref{proof:thm3}. We observe that the first dominant term in the expansion of ${\rm AMSE}(\lambda_{*,q},q,\sigma_w)$  is exactly the same for all values of  $q$, including $q=1$ and is equal to $\sigma_w^2 /(1-1/\delta)$. This is also the same as the AMSE of the OLS. We may consider this term as the `phase transition' term, since it will go to zero only when $\delta>1$. In a nutshell, the first term in the expansion provides the phase transition information. However, we are able to derive the second order term for ${\rm AMSE}(\lambda_{*,q},q,\sigma_w)$. This term gives us what is beyond phase transition analysis. The impact of the signal distribution $p_{\beta}$ and the regularizer $\ell_q$, that is omitted in PT diagram, is revealed in the second order term. As a result, to compare the performance of LQLS with different values of $q$ in the low noise regime, we can compare their second order terms. 

First note that all the regularizers that are studied in Theorem \ref{thm:densebbiggerthanmu} improve the performance of OLS. When the distribution of the coefficients is bounded away from $0$, no significant gain is obtained from LASSO since the second dominant term in the expansion of AMSE is exponentially small. However, the rate of the second order term exhibits an interesting transition from exponential to a polynomial decay when $q$ increases from 1. In fact, it seems that bridge regularizers with $q>1$ offer more substantial improvements over OLS. Even though LASSO is suboptimal, it is not clear which value of $q$ provides the best performance here.  
 Among other LQLS with $q\in (1,2]$, the optimality is determined by the constant involved in the second order term (they all have the same orders). To simplify our discussions, define 
\[
C_q= \frac{(q-1)^2(\mathbb{E}|B|^{q-2})^2}{\mathbb{E}|B|^{2q-2}}, \quad q^*=\argmax_{1<q\leq 2}C_q.
\]
Then LQLS with $q=q^*$ will perform the best. To provide some insights on $q^*$, we focus on a special family of distributions. 

\begin{lemma}\label{lemma:twopoint}
Consider the two-point mixture $|B| \sim \alpha\Delta_{\mu_1}+(1-\alpha)\Delta_{\mu_2}$, where $\Delta_{\mu}$ denotes the probability measure putting mass 1 at $\mu$, $0<\mu_1\leq \mu_2, \alpha \in (0,1)$. Then $q^*=2$ when $\mu_1=\mu_2$, and $q^* \rightarrow 1$ as $\mu_2/\mu_1\rightarrow \infty$. 
\end{lemma}
\begin{proof}
When $\mu_1=\mu_2$, it is clear that $C_q=(q-1)^2\mu_1^{-2}$ and thus $q^*=2$. We now consider $0<\mu_1<\mu_2$. Denote $\kappa=\mu_2/\mu_1$. We can then write $C_q$ as:
\begin{eqnarray*}
C_q=\frac{(q-1)^2(\alpha \mu_1^{q-2}+(1-\alpha)\mu_2^{q-2})^2}{\alpha \mu_1^{2q-2}+(1-\alpha)\mu_2^{2q-2}}
=\frac{(q-1)^2(\alpha+(1-\alpha)\kappa^{q-2})^2}{\mu_1^2(\alpha+(1-\alpha)\kappa^{2q-2})}.
\end{eqnarray*}
Define $\bar{q}=1+\frac{1}{\log \kappa}$. We would like to show that for any $\epsilon>0$, $C_{\bar{q}}>\max_{1+\epsilon \leq q \leq 2}C_q$ for $\kappa$ large enough. That will give us $q^*\in (1,1+\epsilon)$ and hence finishes the proof. To show that, note that for any $q\in [1+\epsilon,2], \kappa \geq 1$,
\begin{eqnarray*}
\frac{C_{\bar{q}}}{C_q}&=&\frac{(\alpha+(1-\alpha)\kappa^{\bar{q}-2})^2}{(\alpha+(1-\alpha)\kappa^{q-2})^2} \cdot \frac{\alpha+(1-\alpha)\kappa^{2q-2}}{\alpha+(1-\alpha)\kappa^{2\bar{q}-2}} \cdot \frac{(\bar{q}-1)^2}{(q-1)^2} \\
&\geq&\alpha^2\cdot \frac{\alpha+(1-\alpha)\kappa^{2\epsilon}}{\alpha+(1-\alpha)e^2}\cdot (\bar{q}-1)^2 \geq 
\frac{\alpha^2(1-\alpha)}{\alpha+(1-\alpha)e^2}\cdot \kappa^{2\epsilon}(\log\kappa)^{-2} \rightarrow \infty.
\end{eqnarray*} 
Therefore, $C_{\bar{q}}>\max_{1+\epsilon \leq q \leq 2}C_q$ when $\kappa$ is sufficiently large.
\end{proof}
\vspace{0.2cm}

Lemma \ref{lemma:twopoint} implies that ridge ($q=2$) regularizer is optimal when the two point mixture components coincide, and the optimal value of $q$ will shift towards $1$ as the ratio of the two points goes off to infinity. Intuitively speaking, one would expect ridge to penalize large signals more aggressively than $q<2$. Hence, in cases the signal has a large dynamic range, ridge penalizes the large signal values more and is not expected to outperform other values of $q$. Note that for the two-point mixture signals, the optimal value of $q$ can be arbitrarily close to 1, however LASSO can never be optimal because its second order term is exponentially small. 

Theorem \ref{thm:densebbiggerthanmu} only studied distributions $p_\beta$ that are bounded away from zero. We next study a more informative and practical case where the distribution of $\beta$ has more mass around zero.

\begin{theorem}\label{thm:densegeneralcase}
Consider the region $\delta>1$ and assume $\mathbb{E}|B|^2<\infty$. For any given $q\in (1,2)$, suppose that $\mathbb{P}(|B|\leq t)=O(t^{2-q+\epsilon})$ (as $t \rightarrow 0$) with $\epsilon$ being any positive constant, then as $\sigma_w\rightarrow 0$
\begin{eqnarray*}
{\rm AMSE}(\lambda_{*,q},q,\sigma_w)=\frac{\sigma_w^2}{1-1/\delta}-\frac{\delta^{3}(q-1)^2(\mathbb{E} |B|^{q-2})^2}{(\delta-1)^3  \mathbb{E} |B|^{2q-2}} \sigma_w^4+o(\sigma_w^4).
\end{eqnarray*}
For $q=2$, as $\sigma_w\rightarrow 0$
\[
{\rm AMSE}(\lambda_{*,2},2,\sigma_w)=\frac{\sigma_w^2}{1-1/\delta}-\frac{\delta^{3}}{(\delta-1)^3  \mathbb{E} |B|^2} \sigma_w^4+o(\sigma_w^4).
\]
For $q=1$, suppose that $\mathbb{P}(|B| \leq t)=\Theta(t^\ell)$ with $\ell>0$, then as $\sigma_w\rightarrow 0$
\[
-|\Theta (\sigma_w^{2\ell+2})| \cdot \bigg(\underbrace{\log \log \ldots \log}_{m\  \rm times} \left(\frac{1}{\sigma_w}\right) \bigg)^{\ell}\lesssim {\rm AMSE}(\lambda_{*,1},1,\sigma_w)-\frac{\sigma^2_w}{1-1/\delta} \lesssim - |\Theta(\sigma_w^{2\ell+2})|,
\]
where $m$ can be any natural number.
\end{theorem}

The proof is presented in Section \ref{proof:thm4}. Before discussing the implications of this result, let us mention a few points about the conditions that are imposed in this theorem. Note that  the condition $\mathbb{P}(|B|\leq t)=O(t^{2-q+\epsilon})$ for $q\in (1,2)$ is necessary otherwise $\mathbb{E}|B|^{q-2}$ appearing in the second order term will be unbounded. Intuitively speaking, for every $q \in (1,2)$ even though $\mathbb{P}(|B|\leq t)=O(t^{2-q+\epsilon})$, the probability density function (pdf) of $|B|$ can still go to infinity at zero. However, the condition requires that it should not go to infinity too fast. Now, we would like to explain some of the interesting implications of this theorem. 
\begin{enumerate}
\item Compared to the results of Theorem \ref{thm:densebbiggerthanmu}, we see that the expansion of ${\rm AMSE}(\lambda_{*,q},q,\sigma_w)$ for $q\in (1,2]$ in Theorem \ref{thm:densegeneralcase} remains the same for more general $B$, while the rate of the second order term for LASSO changes to polynomial from exponential. That means LASSO is more sensitive to the distribution of $\beta$ than other LQLS. 
\item The second order term of LASSO becomes smaller as $\ell$ decreases. Note that when $\ell$ decreases the mass of the distribution around zero increases. Hence, Theorem \ref{thm:densegeneralcase} implies that LASSO performs better when the probability mass of the signal concentrates more around zero. This can be well explained by the sparsity promoting feature of LASSO.

\item As in the case $\mathbb{P}(|B|>\mu)=1$, the first dominant term is the same for all $q\in [1,2]$. Hence we have to compare their second order term. For any given $q\in (1,2]$, suppose $\mathbb{P}(|B| \leq t)=\Theta(t^{2-q+\epsilon})$, then the second term of ${\rm AMSE}(\lambda_{*,q},q,\sigma_w)$ is of order $\sigma_w^4$, while that of LASSO is $\Theta (\sigma_w^{6-2q+2\epsilon})$ (ignore the logarithmic factor). Since both terms are negative, we can conclude LASSO performs better than LQLS with that value of $q$ when $\epsilon \in (0, q-1)$, and performs worse when $\epsilon \in (q-1,\infty)$. This observation has an important implication. The behavior of the distribution of $|B|$ around zero is the most important factor in the comparison between LASSO and other LQLS. If the pdf of $|B|$ is zero at zero, then we should not use LASSO and when it goes to infinity LASSO performs better than LQLS with $q>1$ (at least for those values of $q$ for which our theorem is applicable). 

\item Regrading the case where the probability density function of $|B|$ is finite and positive at zero, our calculations of LASSO are not sharp enough to give an accurate comparison between LASSO and other LQLS. However, the comparison of LQLS for different values of $q>1$ will shed more light on the performance of different regularizers in this case. In the following we consider two of the most popular families of distributions and present an accurate comparison among $q\in (1,2]$.
\end{enumerate}

\begin{lemma}\label{lemma:tail12}
Consider $|B|$ with density function $f(b)=\zeta(\tau,q_0)e^{-\tau b^{q_0}}\mathbbm{1}(0\leq b <\infty)$, where $q_0 \in (0,2], \tau>0$ and $\zeta(\tau,q_0)$ is the normalization constant. Then the best bridge estimator is the one that uses $q^*=\max(1,q_0)$. 
\end{lemma}
\begin{proof}
A simple integration by parts yields, for $q\in (1,2]$
\begin{eqnarray*}
\mathbb{E}|B|^{q-2}=\int_0^{\infty}\zeta(\tau,q_0)b^{q-2}e^{-\tau b^{q_0}}db=\frac{\tau q_0}{q-1}\int_0^{\infty}\zeta(\tau,q_0)b^{q+q_0-2}e^{-\tau b^{q_0}}db=\frac{\tau q_0 \mathbb{E}|B|^{q+q_0-2}}{q-1}.
\end{eqnarray*}
Hence $C_q=\tau^2q_0^2\frac{(\mathbb{E}|B|^{q+q_0-2})^2}{\mathbb{E}|B|^{2q-2}}$. We first consider $q_0\in [1,2]$, then 
\begin{eqnarray*}
C_q = \tau^2q^2_0 \frac{[\mathbb{E}(|B|^{q-1}\cdot |B|^{q_0-1})]^2}{\mathbb{E}|B|^{2q-2}}\overset{(a)}{\leq}\tau^2q^2_0 \mathbb{E}|B|^{2q_0-2}=C_{q_0},
\end{eqnarray*}
where $(a)$ is due to Cauchy-Schwarz inequality. So we obtain $q^*=q_0$. Regrading the case $q_0 \in (0,1)$, let $B'$ be an independent copy of $B$. Then for any $q\in [1,2]$
\begin{eqnarray*}
\mathbb{E}|B|^{q+q_0-2}-\mathbb{E}|B|^{q-1}\mathbb{E}|B|^{q_0-1}=\frac{1}{2}\mathbb{E}(|B|^{q-1}-|B'|^{q-1})(|B|^{q_0-1}-|B'|^{q_0-1}) \leq 0.
\end{eqnarray*}
As a result, we can derive
\begin{eqnarray*}
C_q \leq \tau^2q^2_0\frac{(\mathbb{E}|B|^{q+q_0-2})^2}{(\mathbb{E}|B|^{q-1})^2}\leq \tau^2q^2_0 \mathbb{E}(|B|^{q_0-1})^2=C_1.
\end{eqnarray*}
We can then conclude $q^*=1=\max(1,q_0)$. 
\end{proof}
\vspace{0.2cm}

Note that Lemma \ref{lemma:tail12} studies a family of distributions whose probability density function exists and is non-zero at zero, but exhbit very different tail behaviors. As confirmed by this lemma, in this case the tail behavior has an influence on the performance of LQLS. In particular, LQLS with $q=q_0 \in (1,2]$ is optimal for distributions with the exponential decay tail $e^{-\tau b^{q_0}}$. Since $\hat{\beta}(\lambda,q_0)$ can be considered as the maximum a posterior estimate (MAP), our result suggests that MAP offers the best performance in the low noise regime (among the bridge estimators). This is in general not true. See \cite{zheng2015does} for a counterexample in large noise cases. It is also interesting to observe that as the tail becomes heavier than that of Laplacian distribution, the optimal $q^*$ approaches 1. Again, this observation is consistent with the fact that ridge often penalizes large signal values more aggressively than the other estimators. Hence, if the tail of the distribution is light (like Gaussian distributions), then ridge offers the best performance, otherwise, other values of $q$ offer better results. In the next lemma, we further support this claim by considering a special family of distributions with light tails.

\begin{lemma}
Consider $|B|$ follows a uniform distribution with density function $f(b)=\frac{1}{\theta}\mathbbm{1}(0\leq b \leq \theta)$, where $\theta>0$ is the location parameter. Then $q^*=2$.
\end{lemma}
\begin{proof}
It is clear that
\begin{eqnarray*}
\mathbb{E}|B|^{q-2}=\frac{1}{\theta}\int_0^{\theta}b^{q-2}db=\frac{\theta^{q-2}}{q-1}, \quad \mathbb{E}|B|^{2q-2}=\frac{1}{\theta}\int_0^{\theta}b^{2q-2}db=\frac{\theta^{2q-2}}{2q-1}.
\end{eqnarray*}
Hence, $q^*=\argmax_{1<q\leq 2}\frac{(q-1)^2(\mathbb{E}|B|^{q-2})^2}{\mathbb{E}|B|^{2q-2}}=\argmax_{1<q\leq 2}\frac{2q-1}{\theta^2}=2$.
\end{proof}

\subsection{Implications for classical asymptotics} \label{largesample:imp}

In this section, we would like to show that the results we have derived in the previous sections, has, perhaps surprisingly, some interesting connections to the classical asymptotic setting in which $n \rightarrow \infty$, while $p$ is fixed. Our analysis so far has been focused on the high-dimensional setting in which $n/p\rightarrow \delta \in (0,\infty)$. Furthermore, we assumed that the noise variance is small. In the classical asymptotics, it is assumed that the signal-to-noise ratio of each observation is fixed and $n/p \rightarrow \infty$. Note that having more measurements is at the intuitive level equivalent to less noise. Hence, we expect our low-noise sensitivity to have some implications for the classical asymptotics too. Our goal below is to formalize this connection and explain the implications of our low-noise analysis framework for the classical asymptotics.  

Towards that goal, we will consider the scenarios where the sample size $n$ is much larger than the dimension $p$. Analytically, we let $\delta$ go to infinity and calculate the expansions for AMSE in terms of large $\delta$ (similar to what we did in Section \ref{main:low} for low noise). In this section, we write ${\rm AMSE}(\lambda_{*,q},q, \delta)$ for ${\rm AMSE}(\lambda_{*,q},q, \sigma_w)$ to make it clear that the expansion is derived in terms of $\delta$. Before getting to the results, we should clarify an important issue. Recall the definition of a converging sequence in Definition \ref{def:convseq}. It is straightforward to confirm that the signal-to-noise ratio of each measurement  is ${\rm SNR}\propto \frac{\mathbb{E}|B|^2}{\delta \sigma_w^2}$. Hence if we take $\delta \rightarrow \infty$, SNR of each measurement will go to zero and this is inconsistent with the classical asymptotic setting where the SNR is in general assumed to be fixed. To fix this inconsistency, we will scale the noise term and consider a scaled linear model as follows,
\begin{equation}\label{scale:model}
y=X\beta+ \frac{w}{\sqrt{\delta}},
\end{equation}
where $\{X, \beta, w\}$ is the converging sequence specified in Definition \ref{def:convseq}. With the SNR remained a positive constant, this model is well aligned with the classical setting. Again for comparison purposes we start with the ordinary least squares estimate.

\begin{lemma}
Consider the model \eqref{scale:model} and OLS estimate $\hat{\beta}(0,q)$. Then as $\delta \rightarrow \infty$,
\[
{\rm AMSE}(0,q,\delta)=\frac{\sigma_w^2}{\delta}+\frac{\sigma^2_w}{\delta^2}+o(\delta^{-2}).
\]
\end{lemma}

\begin{proof}
This lemma is a simple application of Lemma \ref{ols:property}. Under model \eqref{scale:model}, Lemma \ref{ols:property} shows that ${\rm AMSE}(0,q,\delta)=\frac{\sigma_w^2}{\delta-1}$. As $\delta \rightarrow \infty$, the expansion can be easily verified.  
\end{proof}
\vspace{0.2cm}

We now discuss the bridge estimators with $q \in[1,2]$.

\begin{theorem}\label{large:thm5}
Consider the model \eqref{scale:model}. Suppose that $\mathbb{P}(|B|>\mu)=1$ with $\mu$ being a positive constant and $\mathbb{E}|B|^2<\infty$. Then for $q\in [1,2]$, as $\delta \rightarrow \infty$,
\begin{eqnarray*}
{\rm AMSE}(\lambda_{*,q},q, \delta)=\frac{\sigma_w^2}{\delta}+\frac{\sigma_w^2}{\delta^2}\cdot \frac{\mathbb{E}|B|^{2q-2}-(q-1)^2(\mathbb{E}|B|^{q-2})^2\sigma_w^2}{\mathbb{E}|B|^{2q-2}} +o(\delta^{-2}). 
\end{eqnarray*}
\end{theorem}
The proof can be found in Section \ref{proof:large:thm5}. Since both Theorems \ref{thm:densebbiggerthanmu} and \ref{large:thm5} are concerned with signals that are bounded away from zero, we can compare their results. Again all of the LQLS have the same first dominant term. However, in the large sample regime, the second order term of LASSO is at the same order as that of other LQLS. Interestingly, the comparison of the constant in the second order term is consistent with that in the low noise case. Hence we obtain the same conclusions for two-point mixture distributions. For instance, bridge with $q \in (1,2]$ outperforms OLS and $q=2$ is optimal when all the mass is concentrated at one point. See Lemma \ref{lemma:twopoint} for more information on the comparison of $C_q$. 

 We now discuss the implications of Theorem \ref{large:thm5} for classical asymptotics. In the classical setting where $n\rightarrow \infty$ and $p$ is fixed, the performance of LQLS has been studied in \cite{knight2000asymptotics}. In particular the LQLS estimates were shown to have the regular $\sqrt{n}$ convergence. In our setting, we first let $n/p \rightarrow \delta$ and then $\delta \rightarrow \infty$. If we apply Theorem 2 in \cite{knight2000asymptotics} to \eqref{scale:model}, a straightforward calculation for the asymptotic variance will give us the first dominant term in ${\rm AMSE}(\lambda_{*,q},q, \delta)$. In other words, the classical asymptotic result for LQLS in \cite{knight2000asymptotics} only provides the ``first-order" information regarding mean square error, and it is the same for all the values of $q\in [1,2]$ under optimal tuning. The virtue of our asymptotic framework is to offer the second order term that can be used to evaluate and compare LQLS more accurately. Similar results can be derived when signals have mass around zero, as presented in the next theorem. 

\begin{theorem}\label{large:thm6}
Consider the model introduced in \eqref{scale:model} and assume $\mathbb{E}|B|^2<\infty$. For any given $q\in (1,2)$, suppose that $\mathbb{P}(|B|\leq t)=O(t^{2-q+\epsilon})$ (as $t \rightarrow 0$) with $\epsilon$ being any positive constant, then as $\delta \rightarrow \infty$,
\begin{eqnarray*}
{\rm AMSE}(\lambda_{*,q},q, \delta)=\frac{\sigma_w^2}{\delta}+\frac{\sigma_w^2}{\delta^2}\cdot \frac{\mathbb{E}|B|^{2q-2}-(q-1)^2(\mathbb{E}|B|^{q-2})^2\sigma_w^2}{\mathbb{E}|B|^{2q-2}}+o(\delta^{-2}),
\end{eqnarray*}
for $q=2$, as $\delta \rightarrow \infty$,
\begin{eqnarray*}
{\rm AMSE}(\lambda_{*,q},q, \delta)=\frac{\sigma_w^2}{\delta}+\frac{\sigma_w^2}{\delta^2}\cdot \frac{\mathbb{E}|B|^2-\sigma_w^2}{\mathbb{E}|B|^2}+o(\delta^{-2}),
\end{eqnarray*}
and for $q=1$, suppose $\mathbb{P}(|B|\leq t)=\Theta(t^{\ell})$ with $0<\ell <1$, then as $\delta\rightarrow \infty$,
\begin{eqnarray*}
-|\Theta (\delta^{-\ell-1})| \cdot \big(\underbrace{\log \log \ldots \log}_{m\  \rm times} (\sqrt{\delta}) \big)^{\ell}\lesssim  {\rm AMSE}(\lambda_{*,q},q, \delta)- \frac{\sigma_w^2}{\delta} \lesssim   -|\Theta (\delta^{-\ell-1})|,
\end{eqnarray*}
where $m$ can be any natural number.
\end{theorem}

The proof is presented in Section \ref{proof:large:thm5}. Theorem \ref{large:thm6} can be compared with Theorem \ref{thm:densegeneralcase}. Again we see that the expansion for $q\in (1,2]$ remains the same for more general signals, while the second order term of LASSO becomes order-wise smaller when signals put more mass around zero. For a given $q\in (1,2]$, it is clear that LASSO outperforms this LQLS when $\mathbb{P}(|B|\leq t)=\Theta(t^{2-q+\epsilon})$ with $\epsilon \in (0,q-1)$. This implies that even in the case when $n$ is much larger than $p$, if the underlying signal has many elements of small values, $\ell_1$ regularization will improve the performance, which is characterized by a second order analysis that is not available from the $\sqrt{n}$ convergence result. Regrading the distributions with tail $e^{-\tau b^{q_0}}$, we see that the comparison among $q\in (1,2]$ in the low noise regime carries over. 

The fact that regularization can improve the performance of the maximum likelihood estimate (i.e., OLS in the context of linear regression with Gaussian noise), seems to be contradictory with the classical results that imply MLE is asymptotically optimal under mild regularity conditions. However, note that the optimality of MLE is concerned with the asymptotic variance (equivalently the first order term) of the estimate. Our results show that many estimators share that first order term,  while their actual performance might be different. Second dominant terms provide much more accurate information in these cases.

\section{Proofs of our main results}\label{sec:proof}

\subsection{Notations and Preliminaries} \label{begin:proof}
Throughout the proofs,  $B$ will be a random variable having the probability measure $p_{\beta}$ that appears in the definition of the converging sequence, and $Z$ will refer to a standard normal random variable. We will also use $\phi(\cdot)$ to denote the density function of $Z$ and $F(b)$ to represent the cumulative distribution function of $|B|$. We further define the following useful notations:
\begin{eqnarray}\label{def:risknotation}
R_q(\chi,\sigma)=\mathbb{E}(\eta_q(B/\sigma+Z;\chi)-B/\sigma)^2, \quad \chi^*_q(\sigma)=\arg\min_{\chi\geq 0}R_q(\chi,\sigma),
\end{eqnarray}
where $B$ and $Z$ are independent. Recall the proximal operator function $\eta_q(u;\chi)$. Since we will be using $\eta_q(u;\chi)$ extensively in the later proofs, we present some useful properties of $\eta_q(u;\chi)$ in the next lemma.  Because $\eta_q(u;\chi)$ has explicit forms when $q=1,2$, we focus on the case $1<q<2$. For notational simplicity we may use $\partial_i f(x_1,x_2,\ldots)$ to represent the partial derivative of $f$ with respect to its $i$th argument. 

\begin{lemma} \label{prox:prop}
For $q\in (1,2)$, the function $\eta_q(u;\chi)$ satisfies the following properties.
\begin{itemize}
\item[(i)] $-\eta_q(u;\chi)=\eta_q(-u;\chi)$.
\vspace{-0.2cm}
\item[(ii)] $u=\eta_q(u;\chi)+\chi q(q-1) \eta_q(u;\chi){\rm sign}(u)$.
\vspace{-0.2cm}
\item[(iii)] $\alpha \eta_q(u;\chi)=\eta_q(\alpha u;\alpha^{2-q}\chi)$, \mbox{~for~}$\alpha>0$.
\vspace{-0.2cm}
\item[(iv)] $\frac{\partial \eta_q(u;\chi)}{\partial u}=\frac{1}{1+\chi q(q-1)|\eta_q(u;\chi)|^{q-2}}$
\vspace{-0.2cm}
\item[(v)] $\frac{\partial \eta_q(u;\chi)}{\partial \chi}=\frac{-q|\eta_q(u;\chi)|^{q-1}{\rm sign}(u)}{1+\chi q(q-1)|\eta_q(u;\chi)|^{q-2}}$
\vspace{-0.2cm}
\item[(vi)]  The function $\partial_2 \eta_q(u;\chi)$ is differentiable with respect to $u$. 
\end{itemize}
\end{lemma}

\begin{proof}
Please refer to Lemmas 7, 8 and 10 in  \cite{weng2016overcoming} for the proofs. 
\end{proof}

\vspace{0.2cm}
We next write down the Stein's lemma \citep{stein1981estimation} that we will apply several times in the proofs.  \\

\noindent  \textbf{Stein's lemma.} \emph{Suppose the function $f: \mathbb{R}\rightarrow \mathbb{R}$ is weakly differentiable and $\mathbb{E}|f'(Z)|<\infty$}, then
\[
\mathbb{E}(Zf(Z))=\mathbb{E}f'(Z).
\]

\subsection{Proof of Lemma \ref{ols:property}} \label{ols:proof}

Since $\delta>1$, $\hat{\beta}(0,q)=(X'X)^{-1}Xy$ is well defined with probability 1 for sufficiently large $n$. We first derive ${\rm AMSE}(\lambda, 2,\sigma_w)$ for the Ridge estimate $\hat{\beta}(\lambda,2)=(X'X+\lambda I)^{-1}X'y$, and then obtain the AMSE for OLS by letting $\lambda \rightarrow 0$. According to Theorem 2.1 in \cite{weng2016overcoming}, it is known that for given $\lambda >0$,
\begin{eqnarray*}
{\rm AMSE}(\lambda,2,\sigma_w)=\delta(\sigma^2-\sigma_w^2), 
\end{eqnarray*}
where $\sigma$ is the solution to the following equation:
\begin{eqnarray*}
\sigma^2=\sigma_w^2+\frac{4\chi^2\mathbb{E}|B|^2+\sigma^2}{\delta(1+2\chi)^2}, \quad \lambda=\chi-\frac{\chi}{\delta(1+2\chi)}.
\end{eqnarray*}
After a few calculations we can obtain
\begin{eqnarray}\label{amse:ridge}
{\rm AMSE}(\lambda,2,\sigma_w)=\frac{\delta(4\chi^2\mathbb{E}|B|^2+\sigma^2_w)}{\delta(1+2\chi)^2-1},
\end{eqnarray}
with $\chi=\frac{1-\delta+2\lambda \delta+\sqrt{(\delta-1-2\lambda\delta)^2+8\lambda \delta^2}}{4\delta}$. Clearly ${\rm AMSE}(\lambda,2,\sigma_w)\rightarrow \frac{\sigma_w^2}{1-1/\delta}$ as $\lambda \rightarrow 0$. We now utilize that result to derive AMSE for OLS. According to the identity below
\begin{eqnarray*}
(X'X+\lambda I)^{-1}=(X'X)^{-1}-\lambda \underbrace{(X'X)^{-1}(I+\lambda (X'X)^{-1})^{-1}(X'X)^{-1}}_{H},
\end{eqnarray*}
we have
\begin{eqnarray}\label{amse:ols}
\frac{1}{p}\|\hat{\beta}(0,q)-\beta\|_2^2-\frac{\sigma_w^2}{1-1/\delta}=\underbrace{\frac{1}{p}\|\hat{\beta}(\lambda,2)-\beta\|_2^2-\frac{\sigma_w^2}{1-1/\delta}}_{J_1}+\underbrace{\frac{1}{p}\|\lambda HX'y\|_2^2}_{J_2}+ \underbrace{\frac{2}{p}\langle \hat{\beta}(\lambda,2)-\beta, \lambda H X'y \rangle}_{J_3}
\end{eqnarray}
Let $\sigma_{\min}(X)$ be the smallest non-zero singular values of $X$. It is not hard to confirm that
\begin{eqnarray*}
\|HX'Y\|_2 \leq \|HX'X\beta\|_2+ \| HX'w\|_2\leq \frac{\|\beta\|_2}{\lambda+\sigma^2_{\min}(X)}+\frac{\|w\|_2}{(\lambda+\sigma^2_{\min}(X))\sigma^2_{\min}(X)}.
\end{eqnarray*}
Since $\sigma_{\min} \overset{a.s.}{\rightarrow} 1-\frac{1}{\sqrt{\delta}}>0$ \cite{bai1993limit} and $\beta, w$ belong to the converging sequence defined in \ref{def:convseq}, we can conclude that $J_2=O(\lambda^2), a.s.$. Moreover, we obtain from \eqref{amse:ridge} that almost surely 
\[
J_1=\frac{\delta(4\chi^2\mathbb{E}|B|^2+\sigma^2_w)}{\delta(1+2\chi)^2-1}-\frac{\sigma_w^2}{1-1/\delta}
\]
The results on $J_1, J_2$ imply that $J_3=O(\lambda), a.s.$. Further note that the term on the left hand side of \eqref{amse:ols} does not depend on $\lambda$. Therefore by letting $n\rightarrow \infty$ and then $\lambda \rightarrow 0$ on both sides of \eqref{amse:ols} finishes the proof.

\subsection{Proof of Theorem \ref{thm:densebbiggerthanmu}} \label{proof:thm3}

\subsubsection{Roadmap} \label{sec:roadmap}

Since the proof has several long steps, we lay out the roadmap to help readers navigate through the details. According to Lemma \ref{prox:prop} part (iii) and Theorem \ref{thm:eqpseudolip}, we know
\begin{eqnarray}\label{amse:formula}
{\rm AMSE}(\lambda_{*,q},q,\sigma_w)=\bar{\sigma}^2R_q(\chi^*_q(\bar{\sigma}),\bar{\sigma}),
\end{eqnarray}
where $\bar{\sigma}$ is the unique solution of 
\begin{eqnarray}\label{fixeq:first}
\bar{\sigma}^2=\sigma_w^2+\frac{\bar{\sigma}^2}{\delta}R_q(\chi^*_q(\bar{\sigma}),\bar{\sigma}).
\end{eqnarray}
Note from the above equation that $\bar{\sigma}$ is a function of $\sigma_w$. In the regime $\sigma_w \rightarrow 0$, we will show $\bar{\sigma}\rightarrow 0$. This fact combined with \eqref{amse:formula} tells us that in order to derive the second-order expansion of ${\rm AMSE}(\lambda_{*,q},q,\sigma_w)$ as a function of $\sigma_w$, it is sufficient to characterize the convergence rate of $\bar{\sigma}$ as $\sigma_w\rightarrow 0$ and $R_q(\chi^*_q(\sigma),\sigma)$ as $\sigma \rightarrow 0$. For that purpose, we will first study the convergence rate of $\chi^*_q(\sigma)$ as $\sigma \rightarrow 0$, which will then enables us to obtain the convergence rate of $R_q(\chi^*_q(\sigma),\sigma)$. We then utilize that result and \eqref{fixeq:first} to derive the rate of $\bar{\sigma}$ as $\sigma_w\rightarrow 0$. We give the proof for $1<q\leq 2$ and $q=1$ in Sections \ref{lqls:away:q12} and \ref{lqls:away:q1}, respectively. 

\subsubsection{Proof for the case $1<q \leq 2$} \label{lqls:away:q12}

Due to the explicit form of $\eta_2(u;\chi)=\frac{u}{1+2\chi}$, all the results for $q=2$ in this section can be easily verified. We thus focus the proof on $1<q<2$. 

\begin{lemma}
Let $\chi_q^*(\sigma)$ be the optimal threshold value as defined in \eqref{def:risknotation}. Then $\chi_q^*(\sigma)\rightarrow 0$ as $\sigma \rightarrow 0$.
\end{lemma}

\begin{proof}
The proof is essentially the same as the one for Lemma 17 in \cite{weng2016overcoming}. Hence we do not repeat the arguments here. 
\end{proof}

\begin{lemma}\label{epsilon1:tune}
For $q\in (1,2]$, suppose that $\mathbb{P}(|B|>\mu)=1$ with $\mu$ being a positive constant and $\mathbb{E}|B|^2<\infty$. Then as $\sigma \rightarrow 0$
\begin{eqnarray*}
R_q(C\sigma^q,\sigma)=1+(C^2q^2\mathbb{E}|B|^{2q-2}-2Cq(q-1)\mathbb{E}|B|^{q-2})\sigma^2+o(\sigma^2),
\end{eqnarray*}
where $C$ is any fixed positive constant.
\end{lemma}

\begin{proof}
We aim to derive the convergence rate of $R_q(\chi,\sigma)$ when $\chi=C\sigma^q$. In this proof, we may write $\chi$ to denote $C\sigma^q$ for notational simplicity. According to Lemma \ref{prox:prop} parts (ii)(iv) and Stein's lemma, we have the following formula for $R_q(\chi,\sigma)$:
\begin{eqnarray}
&&\hspace{-1.5cm}R_q(\chi, \sigma) -1 =\mathbb{E}(\eta_q(B/\sigma+Z;\chi)-B/\sigma-Z)^2+2\mathbb{E}Z(\eta_q(B/\sigma+Z;\chi)-B/\sigma-Z) \nonumber \\
&& \hspace{-.1cm}= \underbrace{\chi^2 q^2 \mathbb{E} |\eta_q(B / \sigma +Z ; \chi)  |^{2q-2}}_{S_1} \underbrace{- 2  \chi q(q-1) \mathbb{E} \frac{|\eta_q (B/ \sigma + Z; \chi)|^{q-2}}{ 1+ \chi q (q-1) |\eta_q (B/ \sigma +Z ; \chi)|^{q-2}}}_{S_2}.   \label{epsilon:1}
\end{eqnarray}
It is straightforward to confirm the following
\begin{eqnarray}
\lim_{\sigma \rightarrow 0}  \frac{S_1}{\sigma^{2}}& =& \lim_{\sigma \rightarrow 0} \frac{ \chi^2 q^2 \mathbb{E} |\eta_q (B/\sigma + Z; \chi)|^{2q-2} }{\sigma^{2}} \nonumber \\
&=&  C^2 q^2 \lim_{\sigma \rightarrow 0} \mathbb{E} |\eta_q (B +\sigma Z; \chi \sigma^{2-q})|^{2q-2} =  C^2 q^2  \mathbb{E} |B|^{2q-2}.  \label{epsilon:2}
\end{eqnarray}
The last equality is obtained by Dominated Convergence Theorem (DCT). The condition of DCT holds due to Lemma \ref{prox:prop} part (ii). We now focus on analyzing $S_2$. We obtain 
\begin{eqnarray}
\frac{-S_2}{\sigma^{2}} &=& 2 C \sigma^{q-2}q(q-1)  \mathbb{E} \frac{|\eta_q (B/ \sigma + Z; \chi)|^{q-2}}{ 1+ \chi q (q-1) |\eta_q (B/ \sigma +Z ; \chi)|^{q-2}} \nonumber \\
&=&2  C \sigma^{q-2}q(q-1)  \mathbb{E} \frac{1}{ |\eta_q (|B|/ \sigma +Z ; \chi)|^{2-q} + \chi q (q-1) }  \nonumber \\
&=& \underbrace{2  C\sigma^{q-2} q(q-1)  \int_{\mu}^\infty \int_{-b/\sigma - \mu/(2\sigma)}^{-b/\sigma + \mu/(2\sigma)} \frac{1}{ |\eta_q (b/ \sigma +z ; \chi)|^{2-q} +\chi q (q-1)} \phi(z) dz dF(b)}_{T_1} \nonumber \\
&&\hspace{-1.4cm}  + \underbrace{2 C \sigma^{q-2}q(q-1)  \int_{\mu}^\infty \int_{\mathbb{R} \backslash [-b/\sigma - \mu/(2\sigma),-b/\sigma + \mu/(2\sigma)]} \frac{1}{ |\eta_q (b/ \sigma +z ; \chi)|^{2-q} +\chi q (q-1)} \phi(z) dz dF(b)}_{T_2}.  \nonumber 
\end{eqnarray}
We then consider $T_1$ and $T_2$ separately. For $T_1$, we have
\begin{eqnarray}
T_1 &\leq& 2  C\sigma^{q-2} q(q-1)  \int_{\mu}^{\infty}\int_{-\mu/\sigma - \mu/(2\sigma)}^{-\mu/\sigma + \mu/(2\sigma)} \frac{1}{ \chi q (q-1)} \phi(z) dz dF(b)  \nonumber \\
&\leq& 2\sigma^{-3}\mu\phi(\mu/(2\sigma))   \rightarrow 0, ~{\rm as }~ \sigma \rightarrow 0. \label{epsilon:3}
\end{eqnarray}
Regarding $T_2$, DCT enables us to conclude
\begin{eqnarray}
\lim_{\sigma \rightarrow 0}T_2&=&\lim_{\sigma \rightarrow 0} 2 C \sigma^{q-2}q(q-1) \mathbb{E}\frac{\mathbbm{1}(Z \notin  [-|B|/\sigma - \mu/(2\sigma),-|B|/\sigma + \mu/(2\sigma)])}{ |\eta_q (|B|/ \sigma +Z ; \chi)|^{2-q} +\chi q (q-1)} \nonumber \\
&=& \lim_{\sigma \rightarrow 0} 2 Cq(q-1) \mathbb{E}\frac{\mathbbm{1}(Z \notin  [-|B|/\sigma - \mu/(2\sigma),-|B|/\sigma + \mu/(2\sigma)])}{ |\eta_q (|B| +\sigma Z ; \chi \sigma^{2-q})|^{2-q} +C\sigma^2 q (q-1)}  \nonumber \\
&=&2Cq(q-1)\mathbb{E}|B|^{q-2}.  \label{epsilon:4}
\end{eqnarray}
Note that DCT works here because for small enough $\sigma$, Lemma \ref{prox:prop} parts (iv)(v) implies
\begin{eqnarray*}
\frac{\mathbbm{1}(Z \notin  [-|B|/\sigma - \mu/(2\sigma),-|B|/\sigma + \mu/(2\sigma)])}{ |\eta_q (|B| +\sigma Z ; \chi \sigma^{2-q})|^{2-q} +C\sigma^2 q (q-1)} \leq \frac{1}{|\eta_q(\mu/2;\chi\sigma^{2-q})|^{2-q}}\leq \frac{1}{|\eta_q(\mu/2;1)|^{2-q}}.
\end{eqnarray*}
Combining \eqref{epsilon:1}, \eqref{epsilon:2}, \eqref{epsilon:3} and \eqref{epsilon:4} together completes the proof.
\end{proof}

Lemma \ref{epsilon1:tune} shows that by choosing an appropriate $\chi$ for $\sigma$ small enough, $R(\chi, \sigma)$ is less than $1$. This result will be used to show that $\chi_q^*(\sigma)$ cannot converge to zero too fast.  We then utilize this fact to derive the exact convergence rate of $\chi_q^*(\sigma)$. This is done in the next lemma. 

\begin{lemma}\label{lq12:amse}
Suppose that $\mathbb{P}(|B|>\mu)=1$ with $\mu$ being a positive constant and $\mathbb{E}|B|^2<\infty$, then for $q\in (1,2]$ we have as $\sigma \rightarrow 0$
\begin{eqnarray*}
&&\chi_q^*(\sigma)= \frac{(q-1)\mathbb{E}|B|^{q-2}}{q\mathbb{E}|B|^{2q-2}}\sigma^q+o(\sigma^q), \\
&&R_q(\chi_q^*(\sigma),\sigma)=1-\frac{(q-1)^2(\mathbb{E}|B|^{q-2})^2}{\mathbb{E}|B|^{2q-2}}\sigma^2+o(\sigma^2).
\end{eqnarray*}
\end{lemma}
\begin{proof}
Choosing $\chi=\frac{(q-1)\mathbb{E}|B|^{q-2}}{q\mathbb{E}|B|^{2q-2}}\cdot \sigma^q$ in Lemma \ref{epsilon1:tune}, we have
\begin{eqnarray}\label{risk:final}
\lim_{\sigma\rightarrow 0}\frac{R_q(\chi,\sigma)-1}{\sigma^2}=-\frac{(q-1)^2(\mathbb{E}|B|^{q-2})^2}{\mathbb{E}|B|^{2q-2}} < 0.
\end{eqnarray}
That means for sufficiently small $\sigma$
\[
R_q(\chi_q^*(\sigma),\sigma)\leq R_q(\chi,\sigma)<1=R_q(0,\sigma).
\]
Hence we can conclude that $\chi_q^*(\sigma) >0$ when $\sigma$ is small enough. Moreover, by a slight change of arguments in the proof of Lemma \ref{epsilon1:tune} summarized below:
\vspace{-0.1cm}
\begin{enumerate}
\item the fact $\chi \sigma^{2-q}=o(1)$ used several times in Lemma \ref{epsilon1:tune} still holds here
\vspace{-0.3cm}
\item $\chi\sigma^{2-q}=o(1)$ and $\chi=o(\sigma^{q})$ are sufficient to have $S_1=o(\sigma^2)$
\vspace{-0.3cm}
\item bounding the term $T_1$ in \eqref{epsilon:3} does not depend on $\chi$
\vspace{-0.3cm}
\item $\chi\sigma^{2-q}=o(1)$ and $\chi=o(\sigma^{q})$ are sufficient to obtain $T_2=o(1)$
\end{enumerate}
\vspace{-0.2cm}
 we can show
\begin{eqnarray}\label{contradict:evid}
\lim_{\sigma\rightarrow 0}\frac{R_q(\chi,\sigma)-1}{\sigma^2}=0,
\end{eqnarray}
for $\chi=O(\exp (-c/\sigma))$ with any fixed positive constant $c$. This implies that $\lim_{\sigma \rightarrow 0} \chi_q^*(\sigma)\cdot e^{c/\sigma}=+\infty$ for any $c>0$. Otherwise there exists a sequence $\sigma_n\rightarrow 0$ such that $\chi_q(\sigma_n)e^{c/\sigma_n}=O(1)$. This result combined with \eqref{risk:final} and \eqref{contradict:evid} contradicts with the fact that $\chi=\chi_q^*(\sigma)$ is the minimizer of $R_q(\chi,\sigma)$. We will use the two aforementioned properties of $\chi_q^*(\sigma)$ we have showed so far in the following proof. For notational simplicity, in the rest of the proof we may use $\chi$ to denote $\chi_q^*(\sigma)$ whenever no confusion is caused. Firstly since $\chi_q^*(\sigma)$ is a non-zero finite value, it is a solution of the first order optimality condition $\frac{\partial R_q(\chi, \sigma)}{\partial \chi}=0$, which can be further written out as 
\begin{eqnarray}
0&=& \mathbb{E}((\eta_q(B/\sigma+Z;\chi)-B/\sigma)\partial_2 \eta_q(B/\sigma+Z;\chi) )  \nonumber \\
&\overset{(a)}{=}& \mathbb{E}\frac{-(\eta_q(B/\sigma+Z;\chi)-B/\sigma-Z)q|\eta_q(B/\sigma+Z;\chi)|^{q-1}\mbox{sign}(B/\sigma+Z)}{1+\chi q(q-1)|\eta_q(B/\sigma+Z;\chi)|^{q-2}} \nonumber \\
&&+\mathbb{E} (Z \partial_2 \eta_q (B/\sigma + Z; \chi)) \nonumber  \\ 
&\overset{(b)}{=}&\chi \underbrace{\mathbb{E}\frac{ q^2|\eta_q(B/\sigma+Z;\chi)|^{2q-2}}{1+\chi q(q-1)|\eta_q(B/\sigma+Z;\chi)|^{q-2}}}_{U_1}-\underbrace{\mathbb{E} \frac{ q(q-1) |\eta_q (B/\sigma+Z; \chi)| ^{4-2q}}{ (|\eta_q (B/\sigma+Z; \chi)|^{2-q}+ \chi q(q-1) )^3}}_{U_2} \nonumber \\
&& - \chi \underbrace{ \mathbb{E} \frac{ q^2(q-1) |\eta_q (B/\sigma+Z; \chi)|^{2-q}}{ (|\eta_q (B/\sigma+Z; \chi)|^{2-q}+ \chi q(q-1) )^3}}_{U_3}.  \label{epsilon1:final}
\end{eqnarray}
We have used Lemma \ref{prox:prop} part (v) to derive $(a)$. To obtain (b), we have used the following steps:
\begin{enumerate}
\item We used Lemma \ref{prox:prop} part (ii) to conclude that \\ 
$$\eta_q(B/\sigma+Z;\chi)-B/\sigma-Z = -\chi q|\eta_q(B/\sigma+Z;\chi)|^{q-1} {\rm sign}(B/\sigma+Z).$$
\item We used the expression we derived in Lemma \ref{prox:prop} part (v) for $\partial_2 \eta_q(B/\sigma+Z;\chi)$ and then employed Stein's lemma to simplify $\mathbb{E} (Z \partial_2 \eta_q (B/\sigma + Z; \chi))$. Note that according to Lemma \ref{prox:prop} part (vi), $\partial_2 \eta_q (B/\sigma + Z; \chi)$ is differentiable with respect to its first argument and hence Stein's lemma can be applied. 
\end{enumerate}
We now evaluate the three terms $U_1$, $U_2$ and $U_3$ individually. Our goal is to show the following:
\begin{itemize}
\item[(i)] $\lim_{\sigma \rightarrow 0} \sigma^{2q-2}U_1=q^2\mathbb{E}|B|^{2q-2}$.
\item[(ii)] $\lim_{\sigma \rightarrow 0} \sigma^{q-2} U_2 =  q(q-1)\mathbb{E}|B|^{q-2} $. 
\item[(iii)]$ \lim_{\sigma \rightarrow 0} \sigma^{2q-4} U_3 = q^2(q-1)\mathbb{E}|B|^{2q-4}.$
\end{itemize}
For the term $U_1$, we can apply Dominated Convergence Theorem (DCT)
\begin{eqnarray*}
\lim_{\sigma \rightarrow 0} \sigma^{2q-2}U_1=\mathbb{E}\lim_{\sigma \rightarrow 0} \frac{ q^2|\eta_q(B+\sigma Z;\chi \sigma^{2-q})|^{2q-2}}{1+\chi \sigma^{2-q} q(q-1)|\eta_q(B+\sigma Z;\chi \sigma^{2-q})|^{q-2}}=q^2\mathbb{E}|B|^{2q-2}.
\end{eqnarray*}
We now derive the convergence rate of $U_2$. We have
\begin{eqnarray}\label{eq:i1components}
U_2 &=& \int_{\mu}^{\infty} \int_{-\infty}^{\infty}  \frac{q(q-1) |\eta_q(b/\sigma+z; \chi)|^{4-2q}}{(|\eta_q(b/\sigma +z ; \chi)|^{2-q} + \chi q(q-1) )^3}  \phi(z)dz dF(b) \nonumber \\
&=& \underbrace{\int_{\mu}^{\infty} \int_{-\frac{b}{\sigma}-\frac{\mu}{2 \sigma} }^{-\frac{b}{\sigma} + \frac{\mu}{2 \sigma}}  \frac{q(q-1) |\eta_q(b/\sigma+z; \chi)|^{4-2q}}{(|\eta_q(b/\sigma +z ; \chi)|^{2-q} + \chi q(q-1) )^3} \phi(z)dz dF(b)}_{U_{21}} \nonumber \\
&&\hspace{-0.3cm}  + \underbrace{\int_{\mu}^{\infty} \int_{ z \notin [-\frac{b}{\sigma}-\frac{\mu}{2 \sigma} ,-\frac{b}{\sigma} 
+ \frac{\mu}{2 \sigma}]}  \frac{q(q-1) |\eta_q(b/\sigma+z; \chi)|^{4-2q}}{(|\eta_q(b/\sigma +z ; \chi)|^{2-q} + \chi q(q-1) )^3}  \phi(z)dz dF(b)}_{U_{22}}.
\end{eqnarray}
First note that 
\begin{eqnarray}\label{eq:i3calc}
&&\sigma^{q-2}U_{21}  \leq \sigma^{q-2}\int_{\mu}^{\infty} \int_{-\frac{b}{\sigma}-\frac{\mu}{2 \sigma} }^{-\frac{b}{\sigma} + \frac{\mu}{2 \sigma}}  \frac{q(q-1) (\mu/(2 \sigma))^{4-2q}}{( \chi q(q-1) )^3} \phi(z)dz dF(b) \nonumber \\
&&\leq  \frac{\mu^{5-2q}\phi(\mu/(2 \sigma)) }{\sigma^{7-3q}2^{4-2q}\chi^3q^2(q-1)^2} \rightarrow 0, {\rm ~as~} \sigma \rightarrow 0,
\end{eqnarray}
where the last step is due to the fact that $\lim_{\sigma \rightarrow 0} \chi e^{c/\sigma}=+\infty$. To evaluate $U_{22}$ we first derive the following bounds for small enough $\sigma$
\begin{eqnarray*}
&& \frac{\mathbbm{1}(z \notin [-\frac{b}{\sigma}-\frac{\mu}{2 \sigma} ,-\frac{b}{\sigma} 
+ \frac{\mu}{2 \sigma}])\cdot q(q-1) |\eta_q(b +\sigma z; \chi \sigma^{2-q})|^{4-2q}}{(|\eta_q(b +\sigma z ; \chi\sigma^{2-q})|^{2-q} + \chi  \sigma^{2-q}q(q-1) )^3} \\
&& \leq \frac{q(q-1)}{|\eta_q(\mu/2 ; \chi \sigma^{2-q})|^{2-q}  }  \leq \frac{q(q-1)}{|\eta_q(\mu/2 ; 1)|^{2-q}  }.
\end{eqnarray*}
Hence we are able to apply DCT to obtain
\begin{eqnarray}\label{eq:i4calc}
\lim_{\sigma \rightarrow 0}\sigma^{q-2}U_{22}=q(q-1)\mathbb{E}|B|^{q-2}.
\end{eqnarray}
Combining \eqref{eq:i1components}, \eqref{eq:i3calc}, and \eqref{eq:i4calc} proves the result (ii). We can use similar arguments to show result (iii). Finally, we utilize the convergence results for $U_1,U_2, U_3$ and Equation \eqref{epsilon1:final} to derive
\begin{eqnarray*}
\lim_{\sigma\rightarrow 0}\frac{\chi}{\sigma^q}=\lim_{\sigma\rightarrow 0}\frac{\lim_{\sigma\rightarrow 0} \sigma^{q-2} U_2}{\lim_{\sigma\rightarrow 0} \sigma^{2q-2} U_2- \lim_{\sigma\rightarrow 0} \sigma^{2q-2} U_3}=\frac{(q-1)\mathbb{E}|B|^{q-2}}{q\mathbb{E}|B|^{2q-2}}.
\end{eqnarray*}
Now since we know the exact convergence order of $\chi_q^*(\sigma)$, \eqref{risk:final} shows the exact order of $R_q(\chi_q^*(\sigma),\sigma)$.
\end{proof}

We are in position to derive the second-order expansion of ${\rm AMSE}(\lambda_{*,q},q,\sigma_w)$ as $\sigma_w \rightarrow 0$ for $q\in (1,2]$. According to Equation \eqref{fixeq:first} and the fact that $\chi=\chi^*_q(\bar{\sigma})$ minimizes $R_q(\chi, \chi^*_q(\bar{\sigma}))$,
it is clear that $\delta(\bar{\sigma}^2-\sigma_w^2)\leq \bar{\sigma}^2R_q(0,\bar{\sigma})=\bar{\sigma}^2$, which combined with the condition $\delta>1$ implies $\bar{\sigma}\rightarrow 0$ as $\sigma_w\rightarrow 0$. This result further enables us to conclude from  \eqref{fixeq:first}:
\begin{eqnarray}\label{theorder:chi}
\lim_{\sigma_w\rightarrow 0}\frac{\bar{\sigma}^2}{\sigma_w^2}=\frac{\delta}{\delta-1},
\end{eqnarray}
where we have used $R_q(\chi_q^*(\bar{\sigma}),\bar{\sigma})\rightarrow 1$ from Lemma \ref{lq12:amse}. We finally utilize Lemma \ref{lq12:amse}, Equations \eqref{amse:formula}, \eqref{fixeq:first} and \eqref{theorder:chi} to derive the expansion of ${\rm AMSE}(\lambda_{*,q},q,\sigma_w)$ in the following way:
\begin{eqnarray*}
&&\hspace{-0.5cm}\sigma_w^{-4}\Big({\rm AMSE}(\lambda_{*,q},q,\sigma_w)-\frac{\sigma_w^2}{1-1/\delta}\Big)=\sigma_w^{-4}\Big(  \bar{\sigma}^2R_q(\chi_q^*(\bar{\sigma}),\bar{\sigma})-\frac{\delta}{\delta-1}(\bar{\sigma}^2-\frac{1}{\delta}\bar{\sigma}^2R_q(\chi_q^*(\bar{\sigma}),\bar{\sigma}))  \Big) \\
&&=\frac{\delta}{\delta-1}\cdot \frac{\bar{\sigma}^4}{\sigma_w^4} \cdot \frac{R_q(\chi_q^*(\bar{\sigma}),\bar{\sigma})-1}{\bar{\sigma}^2}\rightarrow \frac{-\delta^3(q-1)^2(\mathbb{E}|B|^{q-2})^2}{(\delta-1)^3 \mathbb{E}|B|^{2q-2}}.
\end{eqnarray*}
This completes the proof of Theorem \ref{thm:densebbiggerthanmu} for $q\in (1,2]$.

\subsubsection{Proof for the case $q=1$} \label{lqls:away:q1}

\begin{lemma}\label{dense:l1}
Suppose that $\mathbb{P}(|B|>\mu)=1$ with $\mu$ being a positive constant and $\mathbb{E}|B|^2<\infty$, then for $q=1$ as $\sigma \rightarrow 0$
\begin{eqnarray*}
\chi_q^*(\sigma)=O(\phi(\mu/\sigma)), \quad R_q(\chi_q^*(\sigma),\sigma)-1= O(\phi^2(\mu/\sigma)).
\end{eqnarray*}
\end{lemma}

\begin{proof}
We first claim that $\chi_q^*(\sigma)\rightarrow 0$ as $\sigma\rightarrow 0$. Otherwise, there exists a sequence $\sigma_n\rightarrow 0$ such that $\chi_q^*(\sigma_n)\rightarrow C>0$ as $n \rightarrow \infty$. And the limit $C$ is finite. Suppose this is not true, then since $\eta_1(u;\chi)={\rm sign}(u)(|u|-\chi)_+$ we can apply Fatou's lemma to conclude
\begin{eqnarray*}
\liminf_{n\rightarrow \infty} R_q(\chi^*_q(\sigma_n),\sigma_n)\geq \mathbb{E}\liminf_{n\rightarrow \infty} (\eta_1(B/\sigma_n+Z;\chi^*_q(\sigma_n))-B/\sigma_n)^2=+\infty,
\end{eqnarray*}
contradicting with the fact $R_q(\chi^*_q(\sigma_n),\sigma_n)\leq R_q(0,\sigma_n)=1$. We  now calculate the following limit:
\begin{eqnarray*}
&&\lim_{n\rightarrow \infty} R_q(\chi^*_q(\sigma_n),\sigma_n)=\lim_{n\rightarrow \infty}\mathbb{E} (\eta_1(B/\sigma_n+Z;\chi^*_q(\sigma_n))-B/\sigma_n-Z)^2 \\
&&+2\lim_{n\rightarrow \infty} \mathbb{E}Z(\eta_1(B/\sigma_n+Z;\chi^*_q(\sigma_n))-B/\sigma_n-Z)+1=C^2+1.
\end{eqnarray*}
The last step is due to Dominated Convergence Theorem (DCT). The condition of DCT can be verified based on the fact $|u-\eta_1(u;\chi)|\leq \chi$. We can also choose a positive constant $\tilde{C}$ smaller than $C$ and use similar argument to obtain $\lim_{n\rightarrow \infty} R_q(\tilde{C},\sigma_n)=\tilde{C}^2+1$. That means $R_q(\tilde{C},\sigma_n)< R_q(\chi^*_q(\sigma_n),\sigma_n)$ when $n$ is large enough. This is contradicting with the fact $\chi=\chi_q^*(\sigma_n)$ minimizes $R_q(\chi,\sigma_n)$.

We next derive the following bounds:
\begin{eqnarray*}
R_q(\chi,\sigma)-1&=&\mathbb{E}(\eta_1(B/\sigma+Z;\chi)-B/\sigma-Z)^2+2\mathbb{E}(Z(\eta_1(B/\sigma+Z;\chi)-B/\sigma-Z))\\
&\overset{(a)}{=}&\mathbb{E}(\eta_1(B/\sigma+Z;\chi)-B/\sigma-Z)^2+2\mathbb{E}(\partial_1 \eta_1(B/\sigma+Z;\chi)-1)\\
&\overset{(b)}{\leq}& \chi^2 - 2\mathbb{E}\int_{-B/\sigma-\chi}^{-B/\sigma+\chi}\phi(z)dz\overset{(c)}{=} \chi^2 - 4\chi \mathbb{E}\phi(-B/\sigma+\alpha\chi).
\end{eqnarray*}
To obtain (a) we used Stein's lemma; note that $\eta_1(u;\chi)$ is a weakly differentiable function of $u$. Inequality (b) holds since $|\eta_1(u; \chi) - u| \leq \chi$. Equality (c) is the result of the mean value theorem and hence $|\alpha|\leq 1$ is dependent on $B$. From the above inequality, it is straightforward to verify that if we choose $\chi=3e^{-1}\mathbb{E}\phi(\sqrt{2}B/\sigma)$, then 
\begin{eqnarray}
R_q(\chi_q^*(\sigma),\sigma) \leq R_q(\chi,\sigma)<1=R_q(0,\sigma), \label{negative:value}
\end{eqnarray}
for small enough $\sigma$. This means the optimal threshold $\chi^*_q(\sigma)$ is a non-zero finite value.  Hence it is a solution to $\frac{\partial R_q(\chi^*_q(\sigma),\sigma)}{\partial \chi}=0$, which further implies (from now on we use $\chi^*$ to represent $\chi^*_q(\sigma)$ for simplicity): 
\begin{eqnarray}
\chi^*=\frac{\mathbb{E}\phi(\chi^*-B/\sigma)+\mathbb{E}\phi(\chi^*+B/\sigma)}{\mathbb{E}\mathbbm{1}(|Z+B/\sigma|\geq \chi^*)}\leq \frac{2\mathbb{E}\phi(\chi^*-|B|/\sigma)}{\mathbb{E}\mathbbm{1}(|Z+B/\sigma|\geq \chi^*)}\leq  \frac{2\phi(\chi^*-\mu/\sigma)}{\mathbb{E}\mathbbm{1}(|Z+B/\sigma|\geq \chi^*)}, \label{l1:epsilon1}
\end{eqnarray}
where the last inequality holds for small values of $\sigma$ due to the condition $\mathbb{P}(|B|>\mu)=1$. Since $\mathbb{E}\mathbbm{1}(|Z+B/\sigma|\geq \chi^*)\rightarrow 1$, as $\sigma \rightarrow 0$ and $\phi(\chi^*-\mu/\sigma)\leq \phi(\mu/(\sqrt{2}\sigma))e^{(\chi^*)^2/2}$, from \eqref{l1:epsilon1} we can first conclude $\chi^*=o(\sigma)$, which in turn (use \eqref{l1:epsilon1} again) implies $\chi^*=O(\phi(\mu/\sigma))$. 

We now turn to analyzing $R_q(\chi^*,\sigma)$:
\begin{eqnarray*}
R_q(\chi^*,\sigma)-1&=&\mathbb{E}(\eta_1(B/\sigma+Z;\chi^*)-B/\sigma-Z)^2+2\mathbb{E}(\partial_1 \eta_1(B/\sigma+Z;\chi^*)-1) \\
&\geq&-2\mathbb{E}\mathbbm{1}(|B/\sigma+Z|\leq \chi^*) \geq -2\int_{-\mu/\sigma-\chi^*}^{-\mu/\sigma+\chi^*}\phi(z)dz\geq -4\chi^*\phi(\mu/\sigma-\chi^*) \\
&\overset{(d)}{\geq}&\frac{-8\phi^2(\chi^*-\mu/\sigma)}{\mathbb{E}\mathbbm{1}(|Z+B/\sigma|\geq \chi^*)}\overset{(e)}{\sim} -8\phi^2(\mu/\sigma),
\end{eqnarray*}
where $(d)$ is due to \eqref{l1:epsilon1} and (e) holds because $\mathbb{E}\mathbbm{1}(|Z+B/\sigma|\geq \chi^*)\rightarrow 1$ and $\chi^*=o(\sigma)$. This result combined with $R_q(\chi^*,\sigma)-1<0$ from \eqref{negative:value} finishes the proof. 
\end{proof}

We are in position to derive the expansion of ${\rm AMSE}(\lambda_{*,1},1,\sigma_w)$. Similarly as in the proof  for $q\in (1,2]$, we can use Lemma \ref{dense:l1} to derive \eqref{theorder:chi} for $q=1$. Then we apply Lemma \ref{dense:l1} again to obtain
\begin{eqnarray*}
&&{\rm AMSE}(\lambda_{*,1},1,\sigma_w)-\frac{\delta}{\delta-1}\sigma_w^2=\bar{\sigma}^2R_q(\chi^*_q(\bar{\sigma}),\bar{\sigma})-\frac{\delta}{\delta-1}(\bar{\sigma}^2-\bar{\sigma}^2R_q(\sigma^*_q(\bar{\sigma}),\bar{\sigma})/\delta) \\
&&=\frac{\delta \bar{\sigma}^2(R_q(\chi^*_q(\bar{\sigma}),\bar{\sigma})-1)}{\delta-1}=o({\rm exp}(-\bar{\mu}^2/\bar{\sigma}^2))=o({\rm exp}(-\tilde{\mu}^2(\delta-1)/(\delta \sigma_w^2))),
\end{eqnarray*}
where $0<\tilde{\mu}<\bar{\mu}<\mu$. This closes the proof. 

\subsection{Proof of Theorem \ref{thm:densegeneralcase}} \label{proof:thm4}
Similar to the proof of Theorem \ref{thm:densebbiggerthanmu}, we consider two cases, i.e. $1<q  \leq 2$ and $q=1$, and prove them separately. We will follow closely the roadmap illustrated in Section \ref{sec:roadmap}.

\subsubsection{Proof for the case $1<q < 2$}
Again all the results in this section can be proved easily for $q=2$. We will only consider $1<q<2$. Before we start the proof of our main result, we mention a simple lemma that will be used multiple times in our proof. 

\begin{lemma}\label{lem:asymptoticbehavior}
Let $T(\sigma)$ and $\chi(\sigma)$ be two nonnegative sequences with the property: $\chi(\sigma) T^{q-2}(\sigma) \rightarrow 0$, as $\sigma \rightarrow 0$. Then,
\[
\lim_{\sigma \rightarrow 0} \frac{\eta_q(T(\sigma) , \chi(\sigma))}{T(\sigma)} = 1. 
\]
\end{lemma}
\begin{proof}
The proof is a simple application of scale invariance property of $\eta_q$, i.e, Lemma \ref{prox:prop} part (iii). We have
\begin{eqnarray*}
\lim_{\sigma \rightarrow 0} \frac{\eta_q(T(\sigma) , \chi(\sigma))}{T(\sigma)} = \lim_{\sigma \rightarrow 0}  \eta_q(1; \chi(\sigma) T^{q-2} (\sigma)) =1,
\end{eqnarray*}
where the last step is the result of Lemma \ref{prox:prop} part (ii). 
\end{proof}

Our first goal is to show that when $\chi = C \sigma^q$, then $\lim_{\sigma \rightarrow 0} \frac{R(\chi, \sigma) -1}{\sigma^2}$ is a negative constant by choosing an appropriate $C$. However, since this proof is long, we break it to several steps. These steps are summarized in Lemmas \ref{lem:firstdivisionconstantinterval}, \ref{lem:limitzerocloseto1power}, and \ref{lem:funcupperbound1}. Then in Lemma \ref{epsilon1:tune:lq:general} we employ these three results to show that if $\chi = C \sigma^q$, then
\begin{eqnarray*}
\lim_{\sigma \rightarrow 0}\frac{R(\chi,\sigma)-1}{\sigma^2}=C^2q^2\mathbb{E}|B|^{2q-2}-2Cq(q-1)\mathbb{E}|B|^{q-2}.
\end{eqnarray*}

\begin{lemma}\label{lem:firstdivisionconstantinterval}
For any given $q\in (1,2)$, suppose that $\mathbb{P}(|B| < t) = O(t^{2-q+\epsilon})$ (as $t\rightarrow 0$) with $\epsilon$ being any positive constant, $\mathbb{E}|B|^2 < \infty$ and $\chi = C \sigma^q$, where $C>0$ is a fixed number. Then we have
\[
\sigma^{q-2} \int_0^\infty \int_{\frac{-b}{\sigma} - \alpha}^{\frac{-b}{\sigma} + \alpha} \frac{1}{|\eta_q(b/\sigma+z ; \chi)|^{2-q} + \chi q(q-1)} \phi(z)dz dF(b) \rightarrow 0,
\]
as $\sigma \rightarrow 0$. Note that $\alpha$ is an arbitrary positive constant.
\end{lemma}
\begin{proof}
The main idea of the proof is to break this integral into several pieces and prove that each piece converges to zero. Throughout the proof, we will choose $\epsilon$ small enough to be in $(0,q-1)$. Based on the value of $q$, we consider the following intervals. First find the unique non-negative integer of $m^*$  such that 
\[
q \in [2- (\epsilon/(\epsilon+q-1))^{\frac{1}{m^*+1}}, 2-  (\epsilon/(\epsilon+q-1))^{\frac{1}{m^*}} ). 
\]
Denote $\mathcal{S}_m^n(l)=l^m+l^{m+1}+\cdots+l^n (m\leq n)$. Now we define the following intervals: 
\begin{eqnarray}\label{eq:intervaldefsI0_I_i}
\mathcal{I}_{-1} &=& \left[-\frac{b}{\sigma}- \frac{\sigma^{q-\epsilon}}{\log(\frac{1}{\sigma})},  -\frac{b}{\sigma}+ \frac{\sigma^{q-\epsilon}}{\log(\frac{1}{\sigma})} \right],  \nonumber \\
\mathcal{I}_i &=& \left[ -\frac{b}{\sigma} - \frac{\sigma^{ \frac{\epsilon+q-1}{q-1}(2-q)^i-\frac{\epsilon}{q-1}}}{(\log (1/\sigma))^{\mathcal{S}_0^i(2-q)}}, -\frac{b}{\sigma} + \frac{\sigma^{\frac{\epsilon+q-1}{q-1}(2-q)^i-\frac{\epsilon}{q-1}}}{(\log (1/\sigma))^{\mathcal{S}_0^i(2-q)}} \right], \ \ \ 0 \leq i \leq m^*, \nonumber \\
\mathcal{I}_{m^*+1} &=& \left[ -\frac{b}{\sigma} - \frac{1}{(\log (1/\sigma))^{\mathcal{S}_0^{m^*+1}(2-q)}}, -\frac{b}{\sigma} + \frac{1}{(\log (1/\sigma))^{\mathcal{S}_0^{m^*+1}(2-q)}} \right],  \nonumber \\
\mathcal{I}_{m^*+2} &=& \left[ \frac{-b}{\sigma} -\alpha , -\frac{b}{\sigma} +\alpha \right]. 
\end{eqnarray}
We see that for small enough $\sigma$, these intervals are nested: $\mathcal{I}_{-1}\subset \mathcal{I}_0\subset \mathcal{I}_1\subset \ldots \subset \mathcal{I}_{m^*+2}$. Further define 
\begin{eqnarray*}
P_{-1} &=& \sigma^{q-2} \int_0^\infty \int_{\mathcal{I}_{-1}} \frac{1}{|\eta_q(b/\sigma+z ; \chi)|^{2-q} + \chi q(q-1)} \phi(z)dz dF(b), \nonumber \\
P_i &=& \sigma^{q-2} \int_0^\infty \int_{\mathcal{I}_i\backslash \mathcal{I}_{i-1}} \frac{1}{|\eta_q(b/\sigma+z ; \chi)|^{2-q} + \chi q(q-1)} \phi(z)dz dF(b), \quad 0 \leq i \leq m^*+2.
\end{eqnarray*}
Using these notations we have
\begin{eqnarray}\label{eq:allpiecesfirstlemmaeps1qless2}
\sigma^{q-2}\int_0^\infty \int_{\frac{-b}{\sigma} - \alpha}^{\frac{-b}{\sigma} + \alpha} \frac{1}{|\eta_q(b/\sigma+z ; \chi)|^{2-q} + \chi q(q-1)} \phi(z)dz dF(b)=\sum_{i=-1}^{m^*+2}P_i.
\end{eqnarray}
Our goal is to show that $P_i \rightarrow 0$ as $\sigma \rightarrow 0$. Since these intervals have different forms, we consider five different cases (i) $i=-1$, (ii) $i=0$, (iii) $1 \leq i \leq m^*$, (iv) $i=m^*+1$, and (v) $i=m^*+2$ and for each case we show that $P_i \rightarrow 0$.  Let $|\mathcal{I}|$ denote the Lebesgue measure of an interval $\mathcal{I}$. For the first term, we have for a positive constant $\tilde{C}_{-1}$,
\begin{eqnarray}\label{eq:integraloverI0first}
&&\hspace{-0.5cm}P_{-1} \leq  \sigma^{q-2}  \int_0^\infty \int_{\mathcal{I}_{-1}} \frac{1}{\chi q(q-1)} \phi(z)dz dF(b)  \leq \sigma^{q-2}  \int_0^{ \tilde{C}_{-1} \sigma \sqrt{\log(1/\sigma)}  } \int_{\mathcal{I}_{-1}} \frac{1}{\chi q(q-1)} \phi(z)dz dF(b) \nonumber \\
&&\hspace{0.4cm}+ \sigma^{q-2}  \int_{\tilde{C}_{-1} \sigma \sqrt{\log(1/\sigma)}  }^\infty \int_{\mathcal{I}_{-1}} \frac{1}{\chi q(q-1)} \phi(z)dz dF(b) \nonumber \\
&&\hspace{-0.5cm} \leq  \frac{\sigma^{q-2} \phi(0) |\mathcal{I}_{-1}| \mathbb{P} (|B| \leq  \tilde{C}_{-1} \sigma \sqrt{\log(1/\sigma)}   )}{\chi q(q-1)} +  \frac{\sigma^{q-2} \phi(\tilde{C}_{-1}\sqrt{\log(1/\sigma)}  - \frac{\sigma^{q-\epsilon}}{\log(1/\sigma)}) |\mathcal{I}_{-1}|}{\chi q(q-1)} \nonumber \\
&& \hspace{-0.5cm} \leq O(1) \frac{\sigma^{q-\epsilon-2}\mathbb{P} (|B| \leq  \tilde{C}_{-1} \sigma \sqrt{\log(1/\sigma)}   )}{\log(1/\sigma)}+  O(1) \frac{\sigma^{q-\epsilon-2}\phi( \frac{\tilde{C}_{-1}}{2}\sqrt{\log(1/\sigma)})}{\log(1/\sigma)}  \nonumber \\
&& \hspace{-0.5cm} \leq O(1) (\log(1/\sigma))^{\frac{-q+\epsilon}{2}} + O(1) \frac{\sigma^{q-\epsilon-2+\tilde{C}^2_{-1}/8}}{\log(1/\sigma)} \rightarrow 0,
\end{eqnarray}
where we have used the condition $\mathbb{P}(|B|<t)=O(t^{2-q+\epsilon})$ to obtain the last inequality and the last statement holds by choosing $\tilde{C}_{-1}$ large enough. We next analyze the term $P_0$. For a constant $\tilde{C}_{0}>0$ we have
\begin{eqnarray} \label{refereq:p0}
P_0  & \leq& \sigma^{q-2} \int_0^\infty \int_{\mathcal{I}_0 \backslash \mathcal{I}_{-1}} \frac{1}{|\eta_q(b/\sigma+z ; \chi)|^{2-q} } \phi(z)dz dF(b)\nonumber \\
&=&  \sigma^{q-2}  \int_0^{ \tilde{C}_0 \sigma \sqrt{\log(1/\sigma)}  } \int_{\mathcal{I}_0 \backslash \mathcal{I}_{-1}} \frac{1}{|\eta_q(b/\sigma+z ; \chi)|^{2-q} } \phi(z)dz dF(b) \nonumber \\
&&+ \sigma^{q-2}  \int_{ \tilde{C}_0 \sigma \sqrt{\log(1/\sigma)}  }^{\infty} \int_{\mathcal{I}_0 \backslash \mathcal{I}_{-1}} \frac{1}{|\eta_q(b/\sigma+z ; \chi)|^{2-q} } \phi(z)dz dF(b)  \nonumber \\
&&\hspace{-1.4cm} \leq \frac{\sigma^{q-2}\phi(0) |\mathcal{I}_0| \mathbb{P} (|B| < \tilde{C}_0 \sigma \sqrt{\log(1/\sigma)}  ) }{\eta_q^{2-q} \big( \frac{\sigma^{q-\epsilon}}{\log(1/\sigma) }  ; \chi \big)} + \frac{\sigma^{q-2}\phi \Big( \tilde{C}_0  \sqrt{\log(1/\sigma)}  - \frac{\sigma}{\log(1/\sigma) } \Big) |\mathcal{I}_0|  }{\eta_q^{2-q} \big( \frac{\sigma^{q-\epsilon}}{\log(1/\sigma) }  ; \chi \big)}. 
\end{eqnarray}
We have used the fact that $|b/\sigma+z|\geq \frac{\sigma^{q-\epsilon}}{\log(1/\sigma) }$ for $z\notin \mathcal{I}_{-1}$ in the last step. Note that according to Lemma \ref{lem:asymptoticbehavior}, since $(\frac{\sigma^{q-\epsilon}}{\log(1/\sigma)})^{q-2}\chi \propto \sigma^{q^2-(1+\epsilon)q+2\epsilon}(\log(1/\sigma))^{2-q}\rightarrow 0$, we obtain
\begin{eqnarray*}
\lim_{\sigma\rightarrow 0} \frac{\frac{\sigma^{q-\epsilon}}{\log(1/\sigma)}}{\eta_q(\frac{\sigma^{q-\epsilon}}{\log(1/\sigma)};\chi)}=1.
\end{eqnarray*}
With the above result, it is clear that the second term of the upper bound in \eqref{refereq:p0} vanishes if choosing sufficiently large $\tilde{C}_0$. Regarding the first term we know
\begin{eqnarray*}
\frac{\sigma^{q-2} |\mathcal{I}_0| \mathbb{P} (|B| < \tilde{C}_0 \sigma \sqrt{\log(1/\sigma)}  ) }{\eta_q^{2-q} \big( \frac{\sigma^{q-\epsilon}}{\log(1/\sigma) }  ; \chi \big)} \propto \sigma^{q^2-(\epsilon+2)q+3\epsilon+1}(\log(1/\sigma))^{\frac{\epsilon+4-3q}{2}}=o(1).
\end{eqnarray*}
Now we consider an arbitrary $1\leq i \leq m^*$ and show that $P_i \rightarrow 0$. Similarly as bounding $P_0$ we can have
\begin{eqnarray}\label{eq:I_iintervalcomponents}
P_i &\leq&  \sigma^{q-2}  \int_0^{ \tilde{C}_i \sigma \sqrt{\log(1/\sigma)}  } \int_{\mathcal{I}_i\backslash \mathcal{I}_{i-1}} \frac{1}{|\eta_q(b/\sigma+z ; \chi)|^{2-q} } \phi(z)dz dF(b) \nonumber \\
&&+ \sigma^{q-2}  \int_{ \tilde{C}_i \sigma \sqrt{\log(1/\sigma)}  }^{\infty} \int_{\mathcal{I}_i\backslash \mathcal{I}_{i-1}} \frac{1}{|\eta_q(b/\sigma+z ; \chi)|^{2-q} } \phi(z)dz dF(b)  \nonumber \\
&&\hspace{-1.4cm} \leq \frac{\sigma^{q-2}\phi(0) |\mathcal{I}_i| \mathbb{P} (|B| < \tilde{C}_i \sigma \sqrt{\log(1/\sigma)}  ) }{\eta_q^{2-q} \left( \frac{\sigma^{\frac{\epsilon+q-1}{q-1}(2-q)^{i-1}-\frac{\epsilon}{q-1}}}{(\log (1/\sigma))^{\mathcal{S}_0^{i-1}(2-q)}}  ; \chi \right)} + \frac{\sigma^{q-2}\phi \Big( \tilde{C}_i  \sqrt{\log(1/\sigma)}-\frac{\sigma^{\frac{\epsilon+q-1}{q-1}(2-q)^i-\frac{\epsilon}{q-1}}}{(\log (1/\sigma))^{\mathcal{S}_0^i(2-q)}} \Big) |\mathcal{I}_i|  }{\eta_q^{2-q} \left( \frac{\sigma^{\frac{\epsilon+q-1}{q-1}(2-q)^{i-1}-\frac{\epsilon}{q-1}}}{(\log (1/\sigma))^{\mathcal{S}_0^{i-1}(2-q)}}  ; \chi \right)}. 
\end{eqnarray}
We then use Lemma \ref{lem:asymptoticbehavior} to conclude for $i \geq 1$
 \begin{eqnarray}\label{eq:I_iintervalcomponent2}
\lim_{\sigma \rightarrow 0} \frac{ \frac{\sigma^{\frac{\epsilon+q-1}{q-1}(2-q)^{i}-\frac{\epsilon(2-q)}{q-1}}}{(\log (1/\sigma))^{\mathcal{S}_1^{i}(2-q)}} }{\eta_q^{2-q} \left(  \frac{\sigma^{\frac{\epsilon+q-1}{q-1}(2-q)^{i-1}-\frac{\epsilon}{q-1}}}{(\log (1/\sigma))^{\mathcal{S}_0^{i-1}(2-q)}}  ; \chi \right)} =1.
\end{eqnarray}
The condition of Lemma \ref{lem:asymptoticbehavior} can be verified in the following:
\begin{eqnarray*}
 (\sigma^{\frac{\epsilon+q-1}{q-1}(2-q)^{i-1}-\frac{\epsilon}{q-1}} (\log (1/\sigma))^{-\mathcal{S}_0^{i-1}(2-q)})^{q-2}\chi  \propto \sigma^{-\frac{\epsilon+q-1}{q-1}(2-q)^i+\frac{2-q}{q-1}\epsilon+q}(\log(1/\sigma))^{\mathcal{S}_1^i(2-q)} =o(1),
\end{eqnarray*}
where the last step is due to the fact that $-\frac{\epsilon+q-1}{q-1}(2-q)^i+\frac{2-q}{q-1}\epsilon+q\geq -\frac{\epsilon+q-1}{q-1}(2-q)+\frac{2-q}{q-1}\epsilon+q=2q-2>0$.
Using the result \eqref{eq:I_iintervalcomponent2}, it is straightforward to confirm that if $\tilde{C}_i$ is chosen large enough, the second term in \eqref{eq:I_iintervalcomponents} goes to zero. For the first term,
\begin{eqnarray*}\label{eq:I_iintervalcomponent1}
 \lefteqn{\lim_{\sigma \rightarrow 0}  \frac{\sigma^{q-2}\phi(0) |\mathcal{I}_i| \mathbb{P} (|B| < \tilde{C}_i \sigma \sqrt{\log(1/\sigma)}  ) }{\eta_q^{2-q} \left( \frac{\sigma^{\frac{\epsilon+q-1}{q-1}(2-q)^{i-1}-\frac{\epsilon}{q-1}}}{(\log (1/\sigma))^{\mathcal{S}_0^{i-1}(2-q)}}  ; \chi \right)}} \nonumber \\
 &\overset{(a)}{=}& O(1) \cdot \lim_{\sigma \rightarrow 0} \frac{\sigma^{q-2} \frac{\sigma^{\frac{\epsilon+q-1}{q-1}(2-q)^{i}-\frac{\epsilon}{q-1}}}{(\log (1/\sigma))^{\mathcal{S}_0^{i}(2-q)}}   \sigma^{2-q+\epsilon}(\log(1/\sigma))^{\frac{2-q+\epsilon}{2}} }{\frac{\sigma^{\frac{\epsilon+q-1}{q-1}(2-q)^{i}-\frac{\epsilon(2-q)}{q-1}}}{(\log (1/\sigma))^{\mathcal{S}_1^{i}(2-q)}}  } =  O(1) \cdot \lim_{\sigma\rightarrow 0}(\log(1/\sigma))^{\frac{-q+\epsilon}{2}}=0,
 \end{eqnarray*}
where we have used \eqref{eq:I_iintervalcomponent2} to obtain $(a)$. So far we have showed $\lim_{\sigma \rightarrow 0} \sum_{i=-1}^{m^*} P_i =0$. Our next step is to prove that $P_{m^*+1} \rightarrow 0$.
\begin{eqnarray}\label{eq:intergraloverImstar1}
 P_{m^*+1} &\leq& \sigma^{q-2}  \int_0^{ \tilde{C}_{m^*+1} \sigma \sqrt{\log(1/\sigma)}  } \int_{\mathcal{I}_{m^*+1}\backslash \mathcal{I}_{m^*}} \frac{\phi(z)}{|\eta_q(b/\sigma+z ; \chi)|^{2-q} } dz dF(b) \nonumber \\
&&+ \sigma^{q-2}  \int_{ \tilde{C}_{m^*+1} \sigma \sqrt{\log(1/\sigma)}  }^{\infty} \int_{\mathcal{I}_{m^*+1}\backslash \mathcal{I}_{m^*}} \frac{\phi(z)}{|\eta_q(b/\sigma+z ; \chi)|^{2-q} } dz dF(b)  \nonumber \\
&\leq& \frac{\sigma^{q-2}\phi(0) |\mathcal{I}_{m^*+1}| \mathbb{P} (|B| < \tilde{C}_{m^*+1} \sigma \sqrt{\log(1/\sigma)}  ) }{\eta_q^{2-q} \left( \frac{\sigma^{\frac{\epsilon+q-1}{q-1}(2-q)^{m^*}-\frac{\epsilon}{q-1}}}{(\log (1/\sigma))^{\mathcal{S}_0^{m^*}(2-q)}} ; \chi \right)} \nonumber\\
&&+ \frac{\sigma^{q-2}\phi \left( \tilde{C}_{m^*+1}  \sqrt{\log(1/\sigma)}  - \frac{1}{\log(1/\sigma)^{\mathcal{S}_0^{m^*+1}(2-q)}  } \right) |\mathcal{I}_{m^*+1}|  }{\eta_q^{2-q} \left( \frac{\sigma^{\frac{\epsilon+q-1}{q-1}(2-q)^{m^*}-\frac{\epsilon}{q-1}}}{(\log (1/\sigma))^{\mathcal{S}_0^{m^*}(2-q)}} ; \chi \right)}. 
\end{eqnarray}
Again based on \eqref{eq:I_iintervalcomponent2} It is clear that if $\tilde{C}_{m^*+1}$ is large enough, then the second term in \eqref{eq:intergraloverImstar1} goes to zero. We now show the first term goes to zero as well:
\begin{eqnarray*}\label{eq:intergraloverImstar2}
\lefteqn{ \lim_{\sigma \rightarrow 0} \frac{\sigma^{q-2}\phi(0) |\mathcal{I}_{m^*+1}| \mathbb{P} (|B| < \tilde{C}_{m^*+1} \sigma \sqrt{\log(1/\sigma)}  ) }{\eta_q^{2-q} \left( \frac{\sigma^{\frac{\epsilon+q-1}{q-1}(2-q)^{m^*}-\frac{\epsilon}{q-1}}}{(\log (1/\sigma))^{\mathcal{S}_0^{m^*}(2-q)}} ; \chi \right)}} \nonumber \\
 &&\hspace{-1cm} \overset{(b)}{=} O(1) \cdot \lim_{\sigma \rightarrow 0} \frac{\sigma^{q-2}  \frac{1}{\log(1/\sigma)^{\mathcal{S}_0^{m^*+1}(2-q)} }  \sigma^{2-q+\epsilon} (\log(1/\sigma))^{\frac{2-q+\epsilon}{2}}}{ \frac{\sigma^{\frac{\epsilon+q-1}{q-1}(2-q)^{m^*+1}-\frac{\epsilon(2-q)}{q-1}}}{(\log (1/\sigma))^{\mathcal{S}_1^{m^*+1}(2-q)}} }  \\
 &&=O(1)\cdot  \lim_{\sigma \rightarrow 0} \sigma^{\frac{\epsilon-(\epsilon+q-1)(2-q)^{m^*+1}}{q-1}}(\log(1/\sigma))^{\frac{-q+\epsilon}{2}} \overset{(c)}{=} 0,
 \end{eqnarray*}
where $(b)$ holds from  Lemma \ref{lem:asymptoticbehavior} and $(c)$ is due to the condition we imposed on $m^*$ that ensures $(2-q)^{m^*+1}\leq \frac{\epsilon}{\epsilon+q-1}$. The last remaining term of \eqref{eq:allpiecesfirstlemmaeps1qless2} is $P_{m^*+2}$. To prove $P_{m^*+2} \rightarrow 0$, we have
 \begin{eqnarray*}\label{eq:intergraloverImstarplus1first}
 P_{m^*+2} &\leq & \sigma^{q-2} \int_0^{\tilde{C}_{m^*+2} \sigma \sqrt{\log(1/\sigma)}} \int_{\mathcal{I}_{m^*+2}\backslash \mathcal{I}_{m^*+1}} \frac{1}{|\eta_q(b/\sigma+z ; \chi)|^{2-q} } \phi(z)dz dF(b) \nonumber \\
&&+ \sigma^{q-2} \int_{\tilde{C}_{m^*+2} \sigma \sqrt{\log(1/\sigma)}}^\infty \int_{\mathcal{I}_{m^*+2}\backslash \mathcal{I}_{m^*+1}} \frac{1}{|\eta_q(b/\sigma+z ; \chi)|^{2-q} } \phi(z)dz dF(b).
\end{eqnarray*}
By using the same strategy as we did for bounding $P_i~(0\leq i \leq m^*+1)$, the second integral above will go to zero as $\sigma \rightarrow 0$, when $\tilde{C}_{m^*+2}$ is chosen large enough. And the first integral can be bounded by
\begin{eqnarray*}\label{eq:intergraloverImstarplus1second}
 \frac{ \sigma^{q-2} \phi(0)2 \alpha \mathbb{P} (|B| \leq \tilde{C}_{m^*+2} \sigma \sqrt{\log(1/\sigma)} ) }{\eta_q^{2-q} \left( \frac{1}{\log(1/\sigma)^{\mathcal{S}_0^{m^*+1}(2-q)} }  ; \chi \right )} \overset{(d)}{=} O(1) \sigma^{\epsilon} \log(1/\sigma)^{(2-q+\epsilon)/2+\mathcal{S}_1^{m^*+2}(2-q)}  \rightarrow 0,
\end{eqnarray*}
where $(d)$ holds by Lemma \ref{lem:asymptoticbehavior} and the condition of Lemma \ref{lem:asymptoticbehavior} can be easily checked. This completes the proof.
\end{proof}

Define 
\begin{equation}\label{eq:igammadef}
\mathcal{I}^\gamma \triangleq \left [-\frac{b}{\sigma} - \frac{\alpha}{\sigma^{1- \gamma}}, -\frac{b}{\sigma} + \frac{\alpha}{\sigma^{1- \gamma}} \right ].
\end{equation}
 In Lemma \ref{lem:firstdivisionconstantinterval} we proved that:
\[
\sigma^{q-2} \int_0^\infty \int_{\frac{-b}{\sigma} - \alpha}^{\frac{-b}{\sigma} + \alpha} \frac{1}{|\eta_q(b/\sigma+z ; \chi)|^{2-q} + \chi q(q-1)} \phi(z)dz dF(b) \rightarrow 0.
\]
In the next lemma, we would like to extend this result and show that in fact,
\[
\sigma^{q-2} \int_0^\infty \int_{\mathcal{I}^\gamma} \frac{1}{|\eta_q(b/\sigma+z ; \chi)|^{2-q} + \chi q(q-1)} \phi(z)dz dF(b) \rightarrow 0.
\]

\begin{lemma}\label{lem:limitzerocloseto1power}
For any given $q\in (1,2)$, suppose the conditions in Lemma \ref{lem:firstdivisionconstantinterval} hold. Then for any fixed $0<\gamma<1$,
\[
\sigma^{q-2} \int_0^\infty \int_{\mathcal{I}^\gamma \backslash \mathcal{I}_{m^*+2}} \frac{1}{|\eta_q(b/\sigma+z ; \chi)|^{2-q} + \chi q(q-1)} \phi(z)dz dF(b) \rightarrow 0,
\]
as $\sigma \rightarrow 0$. Note that $\mathcal{I}_{m^*+2}$ is defined in \eqref{eq:intervaldefsI0_I_i}.
\end{lemma}
\begin{proof}
As in the proof of Lemma \ref{lem:firstdivisionconstantinterval}, we break the integral into smaller subintervals and prove each one goes to zero. Consider the following intervals:
\[
\mathcal{J}_i= \left [-\frac{b}{\sigma} - \frac{\alpha}{\sigma^{\frac{\epsilon}{1+\theta} \mathcal{S}_0^{i}(1-\epsilon)}},  -\frac{b}{\sigma} + \frac{\alpha}{\sigma^{\frac{\epsilon}{1+\theta} \mathcal{S}_0^{i}(1-\epsilon)}} \right ], 
\]
where $\theta>0$ is an arbitrarily small number and $i$ is an arbitrary natural number. Note that $\{\mathcal{J}_i\}$ is a sequence of nested intervals and $\mathcal{I}_{m^*+2}\subset \mathcal{J}_0$. Our goal is to show that the following integrals go to zero as $\sigma \rightarrow 0$:
\begin{eqnarray}
Q_{-1}&\triangleq&\sigma^{q-2} \int_0^\infty \int_{\mathcal{J}_0 \backslash \mathcal{I}_{m^*+2}} \frac{1}{|\eta_q(b/\sigma+z ; \chi)|^{2-q} + \chi q(q-1)} \phi(z)dz dF(b) \rightarrow 0, \nonumber \\
Q_i &\triangleq& \sigma^{q-2} \int_0^\infty \int_{\mathcal{J}_{i+1} \backslash \mathcal{J}_i} \frac{1}{|\eta_q(b/\sigma+z ; \chi)|^{2-q} + \chi q(q-1)} \phi(z)dz dF(b) \rightarrow 0, \quad i \geq 0. \nonumber 
\end{eqnarray}
Define $\tilde{\sigma}_i \triangleq \frac{1}{\alpha}\sigma^{\frac{\epsilon}{1+\theta} \mathcal{S}_0^i(1-\epsilon)}$. Since $|b/\sigma+z|\geq \alpha$ for $z \notin \mathcal{I}_{m^*+2}$ we obtain
\begin{eqnarray}
Q_{-1}  &\leq&  \sigma^{q-2} \int_0^\infty \int_{\mathcal{J}_0 \backslash \mathcal{I}_{m^*+2}} \frac{1}{|\eta_q(\alpha ; \chi)|^{2-q} } \phi(z)dz dF(b) \nonumber \\
&{=}&  \sigma^{q-2} \int_0^{\frac{\sigma}{\tilde{\sigma}_0} \log(1/\sigma) } \int_{\mathcal{J}_0 \backslash \mathcal{I}_{m^*+2}} \frac{1}{|\eta_q(\alpha ; \chi)|^{2-q} } \phi(z)dz dF(b) \nonumber \\
&&+ \sigma^{q-2} \int_{\frac{\sigma}{\tilde{\sigma}_0} \log(1/\sigma) }^{\infty}  \int_{\mathcal{J}_0 \backslash \mathcal{I}_{m^*+2}} \frac{1}{|\eta_q(\alpha ; \chi)|^{2-q} } \phi(z)dz dF(b)  \nonumber \\
&\leq&\sigma^{q-2} \int_0^{\frac{\sigma}{\tilde{\sigma}_0} \log(1/\sigma) }  \frac{1}{|\eta_q(\alpha ; \chi)|^{2-q} } dF(b) + \sigma^{q-2}  \frac{\phi(\frac{\log(1/\sigma)}{\tilde{\sigma}_0  } - {\frac{1}{\tilde{\sigma}_0}}) |\mathcal{J}_0|}{|\eta_q(\alpha ; \chi)|^{2-q} }. \nonumber
\end{eqnarray}
It is straightforward to notice that the second term above converges to zero. For the first term, by the condition $\mathbb{P}(|B|<t)=O(t^{2-q+\epsilon})$ we derive the following bounds 
\begin{eqnarray*}
&&\sigma^{q-2} \int_0^{\frac{\sigma}{\tilde{\sigma}_0} \log(1/\sigma) }  \frac{1}{|\eta_q(\alpha ; \chi)|^{2-q} } dF(b) \leq O(1) \sigma^{q-2} (\sigma \tilde{\sigma}_0^{-1} \log(1/\sigma))^{2-q+\epsilon} \\
&& = O(1)   \sigma^{\frac{\epsilon(q-1-\epsilon+\theta)}{1+\theta}} (\log(1/\sigma))^{2-q+\epsilon} \rightarrow 0. \nonumber
\end{eqnarray*}
Now we discuss the term $Q_i$ for $i \geq 0$. Similarly as we bounded $Q_{-1}$ we have
\begin{eqnarray}
 Q_i  &\leq& \sigma^{q-2} \int_0^{\frac{\sigma}{\tilde{\sigma}_{i+1}} \log(1/\sigma)} \int_{\mathcal{J}_{i+1} \backslash \mathcal{J}_i} \frac{1}{|\eta_q(\frac{1}{\tilde{\sigma}_i }; \chi)|^{2-q} } \phi(z)dz dF(b) \nonumber \\
 &&+  \sigma^{q-2} \int_{\frac{\sigma}{\tilde{\sigma}_{i+1}} \log(1/\sigma)}^{\infty} \int_{\mathcal{J}_{i+1} \backslash \mathcal{J}_i} \frac{1}{|\eta_q(\frac{1}{\tilde{\sigma}_i} ; \chi)|^{2-q} } \phi(z)dz dF(b). \label{only:one}
 \end{eqnarray}
The second integral in \eqref{only:one} can be easily shown convergent to zero as $\sigma \rightarrow 0$. We now focus on the first integral.
\begin{eqnarray*}
\lefteqn{\sigma^{q-2} \int_0^{\frac{\sigma}{\tilde{\sigma}_{i+1}} \log(1/\sigma)} \int_{\mathcal{J}_{i+1} \backslash \mathcal{J}_i} \frac{1}{|\eta_q(\frac{1}{\tilde{\sigma}_i }; \chi)|^{2-q} } \phi(z)dz dF(b)} \nonumber \\ 
&\leq& \frac{\sigma^{q-2}}{|\eta_q(\frac{1}{\tilde{\sigma}_i }; \chi)|^{2-q}  } \mathbb{P} \big(|B| \leq \frac{\sigma}{\tilde{\sigma}_{i+1}} \log(1/\sigma) \big)  \nonumber \\
&\leq& O(1) \frac{\sigma^{\epsilon} (\log(1/\sigma))^{2-q+\epsilon}}{|\eta_q(\frac{1}{\tilde{\sigma}_i }; \chi)|^{2-q}  \tilde{\sigma}_{i+1}^{2-q+\epsilon}}  \overset{(a)}{=} O(1) \frac{\sigma^{\epsilon} (\log(1/\sigma))^{2-q+\epsilon} \tilde{\sigma}_i^{2-q}}{  \tilde{\sigma}_{i+1}^{2-q+\epsilon}}  \nonumber \\
&=&O(1) \sigma^{\frac{\epsilon(\theta+(q-1-\epsilon)(1-\epsilon)^{i+1})}{1+\theta}} (\log(1/\sigma))^{2-q+\epsilon} =o(1).
\end{eqnarray*}
We have used Lemma \ref{lem:asymptoticbehavior} to obtain $(a)$. Above all we have showed that for any given natural number $i \geq 0$,
\[
\lim_{\sigma \rightarrow 0} \sigma^{q-2} \int_0^\infty \int_{\mathcal{J}_i \backslash \mathcal{I}_{m^*+2}} \frac{1}{|\eta_q(b/\sigma+z ; \chi)|^{2-q} + \chi q(q-1)} \phi(z)dz dF(b)=0.
\]
Now note that as $i$ goes to infinity, the exponent of $\sigma$ in the interval $\mathcal{J}_i$ goes to $\frac{\epsilon}{1+\theta} (1+ (1-\epsilon) + (1-\epsilon)^2+ \ldots) = \frac{1}{1+ \theta}$. So, by choosing small enough $\theta$ and sufficiently large $i$ we can make $\mathcal{I}^{\gamma} \subset \mathcal{J}_i$, hence completing the proof.
\end{proof}

In the last two lemmas, we have been able to prove that for $\chi= C \sigma^q$,
\[
\sigma^{q-2} \int_0^\infty \int_{\mathcal{I}^\gamma} \frac{1}{|\eta_q(b/\sigma+z ; \chi)|^{2-q} + \chi q(q-1)} \phi(z)dz dF(b) \rightarrow 0. 
\]
This result will be used to characterize the following limit 
\[
\lim_{\sigma \rightarrow 0}\sigma^{q-2} \int_0^\infty \int_{\mathbb{R}} \frac{1}{|\eta_q(b/\sigma+z ; \chi)|^{2-q} + \chi q(q-1)} \phi(z)dz dF(b).
\]
Before that we mention a simple lemma that will be applied several times in our proofs.  

\begin{lemma}\label{lem:funcupperbound1}
For $1< q<2$ we have
\[
\frac{1}{|\eta_q (u ; \chi)|^{2-q} + \chi q(q-1)} \leq \frac{2}{|u|^{2-q} (q-1)}. 
\]
\end{lemma}
\begin{proof}
It is sufficient to consider $u>0$. We analyze two different cases:
\begin{enumerate}
\item $\chi \leq u^{2-q} \frac{1}{2q}: $
According to Lemma \ref{prox:prop} part (ii), since we know $\eta_q(u; \chi)  \leq u $, we have 
\[
\eta_q (u ; \chi) = u - \chi q \eta_q^{q-1} (u; \chi)  \geq u - \chi q u^{q-1} \geq u -  u^{2-q} \frac{1}{2q} q \alpha^{q-1} = \frac{u}{2}.
\]
Hence,
\[
\frac{1}{|\eta_q (u ; \chi)|^{2-q} + \chi q(q-1)} \leq \frac{1}{|\eta_q (u ; \chi)|^{2-q} } \leq \frac{2^{2-q}}{u^{2-q}} \leq \frac{2}{(q-1)u^{2-q}}.
\]
\item $\chi \geq u^{2-q} \frac{1}{2q}: $
\[
\frac{1}{|\eta_q (u ; \chi)|^{2-q} + \chi q(q-1)} \leq \frac{1}{\chi q (q-1) } \leq \frac{2}{ (q-1)u^{2-q}}. 
\]
\end{enumerate}
This completes our proof. 
\end{proof}

Now we can consider one of the main results of this section. 

\begin{lemma}\label{epsilon1:tune:lq:general}
For any given $q\in (1,2)$, suppose the conditions in Lemma \ref{lem:firstdivisionconstantinterval} hold. Then for $\chi=C\sigma^q$ we have
\begin{eqnarray*}
\lim_{\sigma \rightarrow 0}\frac{R_q(\chi,\sigma)-1}{\sigma^2}=C^2q^2\mathbb{E}|B|^{2q-2}-2Cq(q-1)\mathbb{E}|B|^{q-2}.
\end{eqnarray*}
\end{lemma}
\begin{proof}
We follow the same roadmap as in the proof of Lemma \ref{epsilon1:tune}. Recall that 
\begin{eqnarray}\label{lq:general:final:zero}
\hspace{-0.2cm} R_q(\chi, \sigma) -1 = \underbrace{\chi^2 q^2 \mathbb{E} |\eta_q(B / \sigma +Z ; \chi)  |^{2q-2}}_{S_1} \underbrace{- 2  \chi q(q-1) \mathbb{E} \frac{|\eta_q (B/ \sigma + Z; \chi)|^{q-2}}{ 1+ \chi q (q-1) |\eta_q (B/ \sigma +Z ; \chi)|^{q-2}} }_{S_2}.
\end{eqnarray}
The first term $S_1$ can be calculated in the same way as in the proof of Lemma  \ref{epsilon1:tune}.
\begin{equation}\label{lq:general:final:one}
\lim_{\sigma \rightarrow 0} \sigma^{-2} S_1 = C^2q^2\mathbb{E}|B|^{2q-2}. 
\end{equation}
We now focus on analyzing $S_2$. First note that restricting $|B|$ to be bounded away from $0$ makes it possible to follow the same arguments used in the proof of Lemma \ref{epsilon1:tune} to obtain,
\begin{eqnarray}\label{lq:general:final:two}
\lim_{\sigma \rightarrow 0}  \mathbb{E} \frac{\mathbbm{1}(|B|>1)}{ |\eta_q (|B|+\sigma Z ; C\sigma^2)|^{2-q} + C\sigma^2 q (q-1) }=\mathbb{E}|B|^{q-2}\mathbbm{1}(|B|>1).
\end{eqnarray}
Hence we next consider the event $|B|\leq 1$. 
\begin{eqnarray*}
&&  \mathbb{E} \frac{\mathbbm{1}(|B|\leq 1)}{ |\eta_q (|B|+\sigma Z ; C\sigma^2)|^{2-q} + C\sigma^2 q (q-1) }  \nonumber \\
&=&   \underbrace{\int_{0}^1 \int_{-b/\sigma - b^c/(2\sigma)}^{-b/\sigma + b^c/(2\sigma)} \frac{1}{ |\eta_q (b +\sigma z ; C\sigma^2)|^{2-q} +C\sigma^2 q (q-1)} \phi(z) dz dF(b)}_{T_1} \nonumber \\
&&  +  \underbrace{\int_{0}^1 \int_{\mathbb{R} \backslash [-b/\sigma - b^c/(2\sigma),-b/\sigma + b^c/(2\sigma)]} \frac{1}{ |\eta_q (b +\sigma z ; C\sigma^2)|^{2-q} +C\sigma^2 q (q-1)} \phi(z) dz dF(b)  }_{T_2},
\end{eqnarray*}
where $c>1$ is a constant that we will specify later. We first analyze $T_2$. Note that,
\begin{eqnarray*}
T_2=\mathbb{E} \frac{\mathbbm{1}(|B+\sigma Z|\geq |B|^c/2, |B|\leq 1)}{ |\eta_q (B +\sigma Z ; C\sigma^2)|^{2-q} +C\sigma^2 q (q-1)},
\end{eqnarray*}
and 
\begin{eqnarray*}
 \frac{\mathbbm{1}(|B+\sigma Z|\geq |B|^c/2, |B|\leq 1)}{ |\eta_q (B +\sigma Z ; C\sigma^2)|^{2-q} +C\sigma^2 q (q-1)}\overset{(a)}{\leq} \frac{2\mathbbm{1}(|B+\sigma Z|\geq |B|^c/2)}{(q-1)|B+\sigma Z|^{2-q}} \leq \frac{|B|^{c(q-2)}}{2^{q-3}(q-1)},
\end{eqnarray*}
where $(a)$ is due to Lemma \ref{lem:funcupperbound1}. For any $1<q<2$, it is straightforward to verify that $\mathbb{E}|B|^{c(q-2)}<\infty$ if $c$ is chosen close enough to $1$. We can then apply Dominated Convergence Theorem (DCT) to obtain 
\begin{eqnarray}\label{lq:general:final:three}
\lim_{\sigma\rightarrow 0}T_2=\mathbb{E}\mathbbm{1}(|B|\geq |B|^c/2, |B|\leq 1)|B|^{q-2}=\mathbb{E}|B|^{q-2}\mathbbm{1}(|B|\leq 1).
\end{eqnarray}
We now turn to bounding $T_1$. According to Lemmas \ref{lem:firstdivisionconstantinterval} and \ref{lem:limitzerocloseto1power}, we know
\[
\sigma^{q-2} \int_0^\infty \int_{\mathcal{I}^\gamma } \frac{1}{|\eta_q(b/\sigma+z ; \chi)|^{2-q} + \chi q(q-1)} \phi(z)dz dF(b) \rightarrow 0,
\]
where $\mathcal{I}^\gamma = [-\frac{b}{\sigma} - \frac{\alpha}{\sigma^{1- \gamma}}, -\frac{b}{\sigma} + \frac{\alpha}{\sigma^{1- \gamma}} ]$. Define $\mathcal{I}_c^\gamma = [-\frac{b}{\sigma} - \frac{b^c}{\sigma^{1- \gamma}}, -\frac{b}{\sigma} + \frac{b^c}{\sigma^{1- \gamma}} ]$ and $\tilde{\mathcal{I}}^c= [-\frac{b}{\sigma}- \frac{b^c}{2\sigma}, -\frac{b}{\sigma}+\frac{b^c}{2 \sigma}]$. For $0\leq b\leq 1$, we get $\mathcal{I}^{\gamma}_c \subseteq \mathcal{I}^{\gamma}$ for any given $\alpha >1$. Therefore,
\begin{eqnarray*}\label{lq:general:final:four}
T_3 \triangleq \sigma^{q-2} \int_0^1 \int_{\mathcal{I}_c^\gamma } \frac{1}{|\eta_q(b/\sigma+z ; \chi)|^{2-q} + \chi q(q-1)} \phi(z)dz dF(b) \rightarrow 0.
\end{eqnarray*}
Hence to bound $T_1$, it is sufficient to bound $T_1-T_3$:
\begin{eqnarray*}
T_1- T_3 &=& \sigma^{q-2} \int_0^1 \int_{\tilde{\mathcal{I}}^c \backslash \mathcal{I}_c^{\gamma}} \frac{1}{|\eta_q(b/\sigma+z ; \chi)|^{2-q} + \chi q(q-1)} \phi(z)dz dF(b) \nonumber \\
& \leq& \sigma^{q-2} \int_0^1 \int_{\tilde{\mathcal{I}}^c \backslash \mathcal{I}_c^{\gamma}} \frac{1}{|\eta_q(b^c/\sigma^{1-\gamma} ; \chi)|^{2-q}  + \chi q(q-1)} \phi(z)dz dF(b) \nonumber \\
&\overset{(b)}{\leq}&  \sigma^{q-2 +(1-\gamma) (2-q)} \int_0^1 \int_{\tilde{\mathcal{I}}^c \backslash \mathcal{I}^{\gamma}_c} \frac{2b^{c(q-2)}}{q-1} \phi(z)dz dF(b) \nonumber \\
&=&   \underbrace{\sigma^{q-2 +(1-\gamma)(2-q)}  \int_0^{\tilde{C}\sigma \sqrt{\log(1/\sigma)}} \int_{\tilde{\mathcal{I}}^c \backslash \mathcal{I}_c^{\gamma}} \frac{2b^{c(q-2)}}{q-1} \phi(z)dz dF(b)}_{T_4} \nonumber \\
&&+  \underbrace{\sigma^{q-2 +(1-\gamma)(2-q)}  \int_{\tilde{C} \sigma \sqrt{\log(1/\sigma)}}^1 \int_{\tilde{\mathcal{I}}^c\backslash \mathcal{I}_c^{\gamma}} \frac{2b^{c(q-2)}}{q-1} \phi(z)dz dF(b)}_{T_5},
\end{eqnarray*}  
where $(b)$ is the result of Lemma \ref{lem:funcupperbound1} and $\tilde{C}$ is a positive constant. We first bound $T_5$ in the following:
\begin{eqnarray*}
T_5 \leq \frac{2\sigma^{q-2 +(1-\gamma)(2-q)}}{q-1} \int_{\tilde{C}  \sigma \sqrt{\log(1/\sigma)}}^1\frac{b^{c(q-1)}}{\sigma}\phi\left(\frac{b}{2\sigma}\right)dF(b)\leq  \frac{2\sigma^{q-3 +(1-\gamma)(2-q)}}{q-1} \phi(\tilde{C}\sqrt{\log(1/\sigma)}/2).
\end{eqnarray*}
It is then easily seen that $T_5$ goes to zero by choosing large enough $\tilde{C}$. For the remaining term $T_4$, 
\begin{eqnarray*}
T_4  &\leq& \frac{2\sigma^{q-3 +(1-\gamma)(2-q)} }{q-1} \int_0^{\tilde{C} \sigma \sqrt{\log(1/\sigma)}} b^{c(q-1)} \phi\left (\frac{b}{2\sigma}\right ) dF(b) \nonumber \\ 
&\leq & \frac{2\sigma^{q-3 +(1-\gamma)(2-q)} }{q-1} (\tilde{C} \sigma \sqrt{\log(1/\sigma)})^{c(q-1)}\phi(0)\mathbb{P}(|B|\leq \tilde{C}  \sigma \sqrt{\log(1/\sigma)})  \nonumber \\ 
&\leq&O(1) \sigma^{c(q-1) - \gamma (2-q)+1-q+\epsilon} (\log(1/\sigma))^{(c(q-1)+2-q+\epsilon)/2} \rightarrow 0. \label{lq:general:final:six}
\end{eqnarray*}
To obtain the last statement, we can choose $\gamma$ close enough to zero and $c$ close to 1. Hence we can conclude $T_1 \rightarrow 0$ as $\sigma \rightarrow 0$. This combined with the results in \eqref{lq:general:final:two} and \eqref{lq:general:final:three}  gives us
\begin{eqnarray*}
\lim_{\sigma \rightarrow 0} -\sigma^{-2}S_2=2Cq(q-1)\mathbb{E}\frac{1}{ |\eta_q (|B|+\sigma Z ; C\sigma^2)|^{2-q} + C\sigma^2 q (q-1)}=2Cq(q-1)\mathbb{E}|B|^{q-2}.
\end{eqnarray*}
The above result together with \eqref{lq:general:final:one} finishes the proof. 
\end{proof}

As stated in the roadmap of the proof, our first goal is to characterize the convergence rate of $\chi^*_q(\sigma)$. Towards this goal, we first show that $\chi^*_q(\sigma)$ cannot be either too large or too small. In particular, in Lemmas \ref{lem:chi*bound1} and \ref{lem:chi*bound2}, we show that $\chi^*_q(\sigma) = O(\sigma^{q-1})$ and $\chi^*_q(\sigma) = \Omega(\sigma^q)$. We then utilize such result in Lemma \ref{lq:general:final:rate:lemma} to conclude that $\chi^*_q(\sigma) = \Theta(\sigma^q)$. 

\begin{lemma}\label{lem:chi*bound1}
Suppose $\mathbb{E}|B|^2<\infty$, if $\chi \sigma^{1-q} = \infty$ and $\chi=o(1)$, then $R_q(\chi, \sigma) \rightarrow \infty$, as $\sigma \rightarrow 0$.
\end{lemma}
\begin{proof}
Consider the formula of $R_q(\chi,\sigma)$ in \eqref{lq:general:final:zero}. Since $\chi=o(1)$, it is straightforward to apply Dominated Convergence Theorem to obtain
\[
\lim_{\sigma \rightarrow 0} \chi^{-2} \sigma^{2q-2} S_1 = q^2 \mathbb{E} |B|^{2q-2}. 
\]
Because $\chi^2\sigma^{2-2q} \rightarrow \infty$, we know $ S_1 \rightarrow  \infty$. Also note
\begin{eqnarray*}
|S_2| \leq 2 \chi q(q-1) \cdot \frac{1}{\chi q(q-1)} =2. 
\end{eqnarray*}
Hence, $R_q(\chi, \sigma )  \rightarrow \infty$. 
\end{proof}

\begin{lemma}\label{lem:chi*bound2}
Suppose that the same conditions for $B$ in Lemma \ref{epsilon1:tune:lq:general} hold, if $\chi = o(\sigma^q)$, then 
\[
\frac{R_q(\chi, \sigma) -1}{\sigma^2} \rightarrow 0, \mbox{~~as~}\sigma \rightarrow 0.
\]
\end{lemma}
\begin{proof}
Consider the expression of $R_q(\chi,\sigma)-1$ in  \eqref{lq:general:final:zero}. First note that
\[
\lim_{\sigma \rightarrow 0} \frac{S_1}{\sigma^2} = \lim_{\sigma \rightarrow 0} \frac{\chi^2 \sigma^{2-2q} q^2 \mathbb{E} |\eta_q(B + \sigma Z; \chi \sigma^{2-q})|^{2q-2}  }{\sigma^2}=0.
\]
Now we study the behavior of $S_2$. Recall that we defined $\mathcal{I}_{-1} = \left[-\frac{b}{\sigma}- \frac{\sigma^{q-\epsilon}}{\log(\frac{1}{\sigma})},  -\frac{b}{\sigma}+ \frac{\sigma^{q-\epsilon}}{\log(\frac{1}{\sigma})} \right]$ and $\mathcal{I}^\gamma = \left [-\frac{b}{\sigma} - \frac{\alpha}{\sigma^{1- \gamma}}, -\frac{b}{\sigma} + \frac{\alpha}{\sigma^{1- \gamma}} \right ]$ in \eqref{eq:intervaldefsI0_I_i} and \eqref{eq:igammadef}, respectively. It is straightforward to use the same argument as for bounding $P_{-1}$ in the proof of Lemma \ref{lem:firstdivisionconstantinterval} (see the derivations in \eqref{eq:integraloverI0first}) to have
\begin{eqnarray*}
\frac{\chi }{\sigma^2} \int_{0}^{\infty} \int_{\mathcal{I}_{-1}} \frac{\phi(z)}{|\eta_q(b/\sigma +z ; \chi)|^{2-q}+ \chi q (q-1)} dz dF(b) \leq \frac{\chi}{\sigma^2} \int_{0}^{\infty} \int_{\mathcal{I}_{-1}} \frac{\phi(z)}{ \chi q (q-1)} dz dF(b) \rightarrow 0.
\end{eqnarray*}
Moreover, since $\chi < C \sigma^q$ for small enough $\sigma$, Lemma \ref{prox:prop} part (v) implies 
\[
|\eta_q (b/\sigma+z; \chi)| \geq |\eta_q (b/\sigma+z; C \sigma^q)|.  
\] 
Therefore, as $\sigma \rightarrow 0$
\begin{eqnarray*}
&&\frac{\chi }{\sigma^2} \int_{0}^{\infty} \int_{\mathcal{I}^{\gamma} \backslash \mathcal{I}_{-1}} \frac{\phi(z)}{|\eta_q(b/\sigma +z ; \chi)|^{2-q} } dz dF(b) \\
&& \leq  \frac{\chi }{\sigma^q} \cdot \frac{1}{\sigma^{2-q}} \int_{0}^{\infty} \int_{\mathcal{I}^{\gamma} \backslash \mathcal{I}_{-1}} \frac{\phi(z)}{|\eta_q(b/\sigma +z ; C \sigma^q)|^{2-q}} dz dF(b)\rightarrow 0,
\end{eqnarray*}
where the last statement holds because of $\frac{1}{\sigma^{2-q}} \int_{0}^{\infty} \int_{\mathcal{I}^{\gamma} \backslash \mathcal{I}_{-1}} \frac{\phi(z)}{|\eta_q(b/\sigma +z ; C \sigma^q)|^{2-q}} dz dF(b)  \rightarrow 0$ that has already been shown in the proof of Lemmas \ref{lem:firstdivisionconstantinterval} and \ref{lem:limitzerocloseto1power}. Above all we have proved
\begin{eqnarray*}
 \frac{\chi}{\sigma^2} \int_{0}^{\infty} \int_{ \mathcal{I}^{\gamma}} \frac{\phi(z)}{|\eta_q(b/\sigma +z ; \chi)|^{2-q}+ \chi q (q-1)} dz dF(b) \rightarrow 0.
\end{eqnarray*}
Based on the above result, we can easily follow the same derivations of bounding the term $T_1$ in the proof of Lemma \ref{epsilon1:tune:lq:general} to conclude
\begin{eqnarray}\label{revise:simple:one}
\lim_{\sigma \rightarrow 0}\frac{\chi}{\sigma^2}\int_{0}^{1}\int_{-b/\sigma-b^c/(2\sigma)}^{-b/\sigma+b^c/(2\sigma)} \frac{\phi(z)}{|\eta_q(b/\sigma +z ; \chi)|^{2-q}+ \chi q (q-1)} dz dF(b)=0.
\end{eqnarray}
Furthermore, because $\chi=o(\sigma^q)$, the analyses to derive Equation \eqref{lq:general:final:two} and bound $T_2$ in the proof of Lemma \ref{epsilon1:tune:lq:general} can be adapted here and yield
\begin{eqnarray}
&&\hspace{-0.4cm} \lim_{\sigma \rightarrow 0}\frac{\chi}{\sigma^2}\int_{1}^{\infty}\int_{-\infty}^{+\infty} \frac{\phi(z)}{|\eta_q(b/\sigma +z ; \chi)|^{2-q}+ \chi q (q-1)} dz dF(b)=0, \label{revise:simple:two} \\
&&\hspace{-0.4cm} \lim_{\sigma \rightarrow 0}\frac{\chi}{\sigma^2}\int_{0}^{1}\int_{\mathbb{R} \backslash [-b/\sigma-b^c/(2\sigma),-b/\sigma+b^c/(2\sigma)]} \frac{\phi(z)}{|\eta_q(b/\sigma +z ; \chi)|^{2-q}+ \chi q (q-1)} dz dF(b)=0. \label{revise:simple:three} 
\end{eqnarray}
Putting results \eqref{revise:simple:one}, \eqref{revise:simple:two} and \eqref{revise:simple:three} together gives us
\begin{eqnarray*}
\lim_{\sigma \rightarrow 0} \frac{-S_2}{\sigma^2}=\lim_{\sigma \rightarrow 0} \frac{2\chi q(q-1)}{\sigma^2} \int_{0}^{\infty} \int_{-\infty}^{\infty} \frac{\phi(z)}{|\eta_q(b/\sigma +z ; \chi)|^{2-q}+ \chi q (q-1)} dz dF(b) =0.
\end{eqnarray*}
This finishes the proof.
\end{proof}
Collecting the results from Lemmas \ref{epsilon1:tune:lq:general}, \ref{lem:chi*bound1} and \ref{lem:chi*bound2}, we can upper and lower bound the optimal threshold value $\chi_q^*(\sigma)$ as shown in the following corollary.

\begin{corollary}\label{lp:general:middle:step}
Suppose the conditions for $B$ in Lemma \ref{epsilon1:tune:lq:general} hold. Then as $\sigma \rightarrow 0$, we have
\[
\chi_q^*(\sigma) = \Omega(\sigma^q), \quad \chi_q^*(\sigma) = O(\sigma^{q-1}).
\] 
\end{corollary}
\begin{proof}
Since $\chi=\chi^*_q(\sigma)$ minimizes $R_q(\chi,\sigma)$, we know
\begin{eqnarray}
&&R_q(\chi^*_q(\sigma),\sigma) \leq R_q(0,\sigma)=1, \mbox{~~for any~}\sigma >0, \label{contrad:fact1}\\
&&\sigma^{-1}(R_q(\chi^*_q(\sigma),\sigma)-1) \leq \sigma^{-2}(R_q(C\sigma^q,\sigma)-1)<-c, \mbox{~~for small enough~}\sigma, \label{contrad:fact2}
\end{eqnarray}
where the last inequality is due to Lemma \ref{epsilon1:tune:lq:general} with an appropriate choice of $C$, and $c$ is a positive constant. Note that we already know $\chi^*_q(\sigma)=o(1)$. If $\chi_q^*(\sigma) \neq O(\sigma^{q-1})$, Lemma \ref{lem:chi*bound1} will contradict with \eqref{contrad:fact1}. If $\chi_q^*(\sigma) \neq \Omega(\sigma^q)$, Lemma \ref{lem:chi*bound2} will contradict with \eqref{contrad:fact2}. 
\end{proof}

We are now able to derive the exact convergence rate of $\chi^*_q(\sigma)$ and $R_q(\chi^*_q(\sigma),\sigma)$.

\begin{lemma}\label{lq:general:final:rate:lemma}
For any given $q\in (1,2)$, suppose the conditions in Lemma \ref{lem:firstdivisionconstantinterval} for $B$ hold. Then we have 
\begin{eqnarray*}
\chi^*_q(\sigma)&=& \frac{(q-1)\mathbb{E}|B|^{q-2}}{q\mathbb{E}|B|^{2q-2}}\sigma^q+o(\sigma^q), \\
R_q(\chi_q^*(\sigma),\sigma) &=& 1- \frac{(q-1)^2(\mathbb{E}|B|^{q-2})^2}{\mathbb{E}|B|^{2q-2}}\sigma^2+o(\sigma^2).
\end{eqnarray*}
\end{lemma}

\begin{proof}
In this proof, we use $\chi^*$ to denote $\chi^*_q(\sigma)$ for notational simplicity. Using the notations in Equation \eqref{epsilon1:final}, we know that $\chi^*$ satisfies the following equation:
\[
0 = \chi^* U_1 - U_2- \chi^* U_3. 
\]
Our first goal is to show that $\sigma^{q-2} U_2 \rightarrow q(q-1) \mathbb{E} |B|^{q-2}$ as $\sigma \rightarrow 0$. Define the interval 
\begin{equation}\label{eq:intervalIminus1}
\mathcal{K} = [-b/\sigma- (\chi^*)^{1/(2-q)} , -b/\sigma+ (\chi^*)^{1/(2-q)} ].
\end{equation}
 Then we have,
\begin{eqnarray}
 \frac{U_2}{\sigma^{2-q}}  &=&  \frac{q(q-1)}{\sigma^{2-q}} \int_{0}^\infty
 \int_{\mathcal{K}} \frac{|\eta_q(b/\sigma+z; \chi^*)|^{4-2q}}{(|\eta_q(b/\sigma+z; \chi^*)|^{2-q}+ \chi^* q (q-1))^3}  \phi(z)dz dF(b)\nonumber \\
 &&\hspace{-0.3cm} + \frac{q(q-1)}{\sigma^{2-q}} \int_{0}^\infty
 \int_{\mathbb{R} \backslash \mathcal{K}} \frac{|\eta_q(b/\sigma+z; \chi^*)|^{4-2q}}{(|\eta_q(b/\sigma+z; \chi^*)|^{2-q}+ \chi^* q (q-1))^3} \phi(z)dz dF(b). \label{lq:general:dense:final:one}
 \end{eqnarray}
We first show that the first term in \eqref{lq:general:dense:final:one} goes to zero. Note that $\eta_q ((\chi^*)^{1/(2-q)} ; \chi^*) =  (\chi^*)^{1/(2-q)}\eta_q(1;1)$ by Lemma \ref{prox:prop} part (iii), we thus have
 \begin{eqnarray}
&&  \frac{q(q-1)}{\sigma^{2-q}} \int_{0}^\infty
 \int_{\mathcal{K}} \frac{|\eta_q(b/\sigma+z; \chi^*)|^{4-2q}}{(|\eta_q(b/\sigma+z; \chi^*)|^{2-q}+ \chi^* q (q-1))^3} \phi(z)dzdF(b) \nonumber \\
 & \leq& \frac{q(q-1)}{\sigma^{2-q}} \int_{0}^\infty  \int_{\mathcal{K}} \frac{|\eta_q((\chi^*)^{1/(2-q)}; \chi^*)|^{4-2q}}{( \chi^* q (q-1))^3}\phi(z)dzdF(b) \nonumber \\
 & \leq& \frac{1}{\sigma^{2-q}} \int_{0}^\infty  \int_{\mathcal{K}} \frac{\eta_q^{4-2q}(1;1)}{ \chi^* (q (q-1))^2}\phi(z)dzdF(b) \nonumber \\
 & \leq& \frac{1}{\sigma^{2-q}} \int_{0}^{C_{1} \sigma \sqrt{\log(1/\sigma)}}  \int_{\mathcal{K}} \frac{\eta_q^{4-2q}(1;1)}{ \chi^* (q (q-1))^2}\phi(z)dzdF(b) \nonumber \\
 &&+ \frac{\eta_q^{4-2q}(1;1)}{q^2(q-1)^2\sigma^{2-q}\chi^*}|\mathcal{K}|\phi(C_1\sqrt{\log(1/\sigma)}-(\chi^*)^{1/(2-q)}). \nonumber
 \end{eqnarray}
 Since we have already shown $\chi^*= \Omega(\sigma^{q})$ in Corollary \ref{lp:general:middle:step}, it is straightforward to see that the second integral in the above bound is negligible for large enough $C_{1}$. For the first term, we know
 \begin{eqnarray*}\label{lq:general:final:rate:four}
 \frac{1}{\chi^*\sigma^{2-q}} \int_{0}^{C_{1} \sigma \sqrt{\log(1/\sigma)}}  \int_{\mathcal{K}}\phi(z)dzdF(b) \leq O(1) (\chi^*)^{(q-1)/(2-q)} \sigma^{\epsilon} (\log(1/\sigma))^{\frac{2-q+\epsilon}{2}}  = o(1). 
 \end{eqnarray*}
Our next goal is to  find the limit of the second term in \eqref{lq:general:dense:final:one}. In order to do that, we again break the integral into several pieces. Recall the intervals $\mathcal{I}^{\gamma}, \mathcal{I}_{-1}, \mathcal{I}_0, \mathcal{I}_1, \ldots,\mathcal{J}_0,\mathcal{J}_1,\ldots$ that we introduced in Lemmas \ref{lem:firstdivisionconstantinterval} and \ref{lem:limitzerocloseto1power}. We consider two different cases:
\begin{enumerate}
\item In this case, we assume that $(\chi^*)^{1/(2-q)} = o(\frac{\sigma^{q-\epsilon}}{\log(1/\sigma)})$. 

Hence $\mathcal{K}\subseteq \mathcal{I}_{-1}$. We have
\begin{eqnarray}
&&\frac{1}{\sigma^{2-q}} \int_{0}^\infty  \int_{\mathcal{I}_{-1} \backslash \mathcal{K}} \frac{|\eta_q(b/\sigma+z; \chi^*)|^{4-2q}}{(|\eta_q(b/\sigma+z; \chi^*)|^{2-q}+ \chi^* q (q-1))^3} \phi(z)dz dF(b) \nonumber \\
 & \leq &  \frac{1}{\sigma^{2-q}} \int_{0}^\infty  \int_{\mathcal{I}_{-1} \backslash \mathcal{K}} \frac{1}{|\eta_q(b/\sigma+z; \chi^*)|^{2-q}} \phi(z)dz dF(b) \nonumber \\
 &\leq&  \frac{1}{\sigma^{2-q}} \int_{0}^\infty  \int_{\mathcal{I}_{-1} \backslash \mathcal{K}} \frac{1}{\chi^*\eta_q^{2-q}(1;1)} \phi(z)dz dF(b) \nonumber \\
 &\leq& \frac{1}{\sigma^{2-q}} \int_{0}^{C_2 \sigma \sqrt{\log(1/\sigma)}}  \int_{\mathcal{I}_{-1} \backslash \mathcal{K}} \frac{1}{\chi^*\eta_q^{2-q}(1;1)} \phi(z)dz dF(b) \nonumber \\
 &&+  \frac{1}{\sigma^{2-q}} \int_{C_2 \sigma \sqrt{\log(1/\sigma)}}^\infty  \int_{\mathcal{I}_{-1} \backslash \mathcal{K}} \frac{1}{\chi^*\eta_q^{2-q}(1;1)} \phi(z)dz dF(b) \nonumber.
\end{eqnarray}
The fact that  $\chi^*(\sigma) = \Omega(\sigma^{q})$ enables us to conclude that the second integral above goes to zero by choosing large enough $C_2$. Regarding the first term we know
\begin{eqnarray*}
\lefteqn{ \frac{1}{\chi^*\sigma^{2-q}} \int_{0}^{C_2 \sigma \sqrt{\log(1/\sigma)}}  \int_{\mathcal{I}_{-1} \backslash \mathcal{K}} \phi(z)dz dF(b)} \nonumber \\
& \leq& \frac{\phi(0) |\mathcal{I}_{-1}|\mathbb{P}(|B|\leq C_2\sigma \sqrt{\log(1/\sigma)})}{\sigma^{2-q} \chi^* } = O(1) \cdot \frac{\sigma^q}{\chi^*} \cdot (\log(1/\sigma))^{\frac{-q+\epsilon}{2}}  \overset{(a)}{=} o(1),
\end{eqnarray*}
where $(a)$ is due to $\chi^*=\Omega(\sigma^q)$. We now consider another integral.
\begin{eqnarray}
&& \frac{1}{\sigma^{2-q}} \int_{0}^\infty  \int_{\mathcal{I}^{\gamma} \backslash \mathcal{I}_{-1}} \frac{|\eta_q(b/\sigma+z; \chi^*)|^{4-2q}}{(|\eta_q(b/\sigma+z; \chi^*)|^{2-q}+ \chi^* q (q-1))^3} \phi(z)dz dF(b) \nonumber \\
&\leq&\frac{1}{\sigma^{2-q}} \int_{0}^\infty  \int_{\mathcal{I}^{\gamma} \backslash \mathcal{I}_{-1}} \frac{1}{|\eta_q(b/\sigma+z; \chi^*)|^{2-q}} \phi(z)dz dF(b). \nonumber
\end{eqnarray}
Our goal is to show that this integral goes to zero as well. We use the following calculations:
\begin{eqnarray}
&&\frac{1}{\sigma^{2-q}} \int_{0}^\infty  \int_{\mathcal{I}^{\gamma} \backslash \mathcal{I}_{-1}} \frac{1}{|\eta_q(b/\sigma+z; \chi^*)|^{2-q}} \phi(z)dz dF(b) \nonumber \\
&\leq&\frac{1}{\sigma^{2-q}}  \sum_{i=0}^{m_*+2} \int_{0}^\infty  \int_{\mathcal{I}_{i} \backslash \mathcal{I}_{i-1}} \frac{1}{|\eta_q(b/\sigma+z; \chi^*)|^{2-q}} \phi(z)dz dF(b) \nonumber \\
&& + \frac{1}{\sigma^{2-q}} \sum_{i=1}^{\ell} \int_{0}^\infty  \int_{\mathcal{J}_{i} \backslash \mathcal{J}_{i-1}} \frac{1}{|\eta_q(b/\sigma+z; \chi^*)|^{2-q}} \phi(z)dz dF(b) \nonumber \\
&&+ \frac{1}{\sigma^{2-q}} \int_{0}^\infty  \int_{\mathcal{J}_{0} \backslash \mathcal{I}_{m^*+2}} \frac{1}{|\eta_q(b/\sigma+z; \chi^*)|^{2-q}} \phi(z)dz dF(b), \nonumber
\end{eqnarray}
where $\ell$ is chosen in a way such that $\mathcal{I}^\gamma \subseteq \mathcal{J}_{\ell}$. Define $m_i =|\mathcal{I}_i| $ and $\tilde{m}_i = |\mathcal{J}_i| $. Note that we did similar calculations for the case $\chi = C \sigma^q$ in Lemmas \ref{lem:firstdivisionconstantinterval} and \ref{lem:limitzerocloseto1power}. The key argument regarding $\chi$ that we used there to show each term above converges to zero was that $\eta_q(m_i; C \sigma^q) = \Theta (m_i)$ and $\eta_q (\tilde{m}_i; C \sigma^q) = \Theta(\tilde{m}_i)$. Hence, if we can show that $\eta_q(m_i ; \chi^*) = \Theta (m_i)$ and $\eta_q(\tilde{m}_i ; \chi^*) = \Theta (\tilde{m}_i)$ in the current case, then those proofs will carry over and we will have 
\[
\frac{1}{\sigma^{2-q}}\int_{0}^\infty  \int_{\mathcal{I}^{\gamma} \backslash \mathcal{I}_{-1}} \frac{1}{|\eta_q(b/\sigma+z; \chi^*)|^{2-q}} \phi(z)dz dF(b) \rightarrow 0. 
\]
For this purpose, we make use of Lemma \ref{lem:asymptoticbehavior}. Note that since $m_{-1}<m_0< m_1< \ldots< m_{m*+2}< \tilde{m}_0<\tilde{m}_1<\tilde{m}_2 \ldots < \tilde{m}_\ell$, we only need to confirm the condition of Lemma \ref{lem:asymptoticbehavior} for $m_{-1}$. We have
\[
\chi^* m_{-1}^{q-2} = \chi^* \sigma^{(q-\epsilon)(q-2)}\log(1/\sigma)^{2-q} =o(1),
\]
by the assumption of Case 1. Hence in the current case we have obtained
\begin{eqnarray*}
 \frac{1}{\sigma^{2-q}} \int_{0}^\infty  \int_{\mathcal{I}^{\gamma}} \frac{|\eta_q(b/\sigma+z; \chi^*)|^{4-2q}}{(|\eta_q(b/\sigma+z; \chi^*)|^{2-q}+ \chi^* q (q-1))^3} \phi(z)dz dF(b) \rightarrow 0.
\end{eqnarray*}
Furthermore, it is clear that
\[
 \frac{\sigma^{q-2} |\eta_q(b/\sigma+z; \chi^*)|^{4-2q}}{(|\eta_q(b/\sigma+z; \chi^*)|^{2-q}+ \chi^* q (q-1))^3}\leq \frac{1}{|\eta_q(b+\sigma z; \chi^*\sigma^{2-q})|^{2-q}+ \chi^* \sigma^{2-q}q (q-1)}. 
\]
We can then follow the same line of arguments for deriving $\lim_{\sigma\rightarrow 0}-S_2/\sigma^2$ in the proof of Lemma \ref{epsilon1:tune:lq:general} to obtain $\lim_{\sigma \rightarrow 0}\sigma^{q-2}U_2=q(q-1)\mathbb{E}|B|^{q-2}$.

\item The other case is $(\chi^*)^{\frac{1}{2-q}} = \Omega(\frac{\sigma^{q-\epsilon}}{\log(1/\sigma)})$. Because $(\chi^*)^{\frac{1}{2-q}} = \Omega(\frac{\sigma^{q-\epsilon}}{\log(1/\sigma)})$ and $\chi^*=O(\sigma^{q-1})$, there exists a value of $0\leq \bar{m} \leq m^*+1$ such that for $\sigma$ small enough, $(\chi^*)^{\frac{1}{2-q}}=o(|\mathcal{I}_{\bar{m}}|)$ and $(\chi^*)^{\frac{1}{2-q}}=\Omega(|\mathcal{I}_{\bar{m}-1}|)$. We then break the integral into:
\begin{eqnarray}\label{eq:threetermsboundingcase2}
 \lefteqn{\frac{1}{\sigma^{2-q}} \int_{0}^\infty  \int_{\mathcal{I}^{\gamma} \backslash \mathcal{K}} \frac{|\eta_q(b/\sigma+z; \chi^*)|^{4-2q}}{(|\eta_q(b/\sigma+z; \chi^*)|^{2-q}+ \chi^* q (q-1))^3} \phi(z)dz dF(b) } \nonumber \\
 &=&  \frac{1}{\sigma^{2-q}} \int_{0}^\infty  \int_{\mathcal{I}_{\bar{m}} \backslash \mathcal{K}} \frac{|\eta_q(b/\sigma+z; \chi^*)|^{4-2q}}{(|\eta_q(b/\sigma+z; \chi^*)|^{2-q}+ \chi^* q (q-1))^3} \phi(z)dz dF(b) \nonumber \\
 &&+ \frac{1}{\sigma^{2-q}} \int_{0}^\infty  \int_{\mathcal{I}^{\gamma} \backslash \mathcal{I}_{\bar{m}}} \frac{|\eta_q(b/\sigma+z; \chi^*)|^{4-2q}}{(|\eta_q(b/\sigma+z; \chi^*)|^{2-q}+ \chi^* q (q-1))^3} \phi(z)dz dF(b)\label{lq:general:final:rate:two}
\end{eqnarray}
Once we show that each of the two integrals above goes to zero as $\sigma \rightarrow 0$, then the subsequent arguments will be exactly the same as the ones in Case 1. Regarding the first integral,
\begin{eqnarray*}
 \lefteqn{ \frac{1}{\sigma^{2-q}} \int_{0}^\infty  \int_{\mathcal{I}_{\bar{m}} \backslash \mathcal{K}} \frac{|\eta_q(b/\sigma+z; \chi^*)|^{4-2q}}{(|\eta_q(b/\sigma+z; \chi^*)|^{2-q}+ \chi^* q (q-1))^3} \phi(z)dz dF(b)} \nonumber \\
 &\leq& \frac{1}{\sigma^{2-q}} \int_{0}^{C_3 \sigma \sqrt{\log(1/\sigma)}}  \int_{\mathcal{I}_{\bar{m}} \backslash \mathcal{K}} \frac{1}{\chi^*\eta^{2-q}_q(1;1)} \phi(z)dz dF(b) \nonumber\\
 &&+  \frac{1}{\sigma^{2-q}} \int_{C_3 \sigma \sqrt{\log(1/\sigma)}}^\infty  \int_{\mathcal{I}_{\bar{m}} \backslash \mathcal{K}} \frac{1}{\chi^*\eta^{2-q}_q(1;1)} \phi(z)dz dF(b)  \nonumber \\
 &\leq& \frac{\phi(0)|\mathcal{I}_{\bar{m}}|\mathbb{P}(|B|\leq C_3\sigma \sqrt{\log(1/\sigma)})}{\sigma^{2-q}\eta^{2-q}_q(1;1)\chi^*} + \frac{\phi(C_3\sqrt{\log(1/\sigma})/2)}{\sigma^{2-q}\eta^{2-q}_q(1;1)\chi^*}.
\end{eqnarray*}
Since $\chi^* = \Omega(\sigma^q)$ from Corollary \ref{lp:general:middle:step}, it is clear that the second term in the above upper bound goes to zero by choosing large enough $C_3$. Regarding the first term we have
\begin{eqnarray*}\label{lq:general:final:rate:one}
 \frac{|\mathcal{I}_{\bar{m}}| \mathbb{P}(|B|\leq C_3\sigma \sqrt{\log(1/\sigma)})}{\sigma^{2-q}\chi^*} &\overset{(a)}{\leq}&  O(1)  \frac{\sigma^{\epsilon} (\log(1/\sigma))^{\frac{2-q+\epsilon}{2}}  \sigma^{\frac{\epsilon+q-1}{q-1}(2-q)^{\bar{m}}-\frac{\epsilon}{q-1}} }{\chi^* \log(1/\sigma)^{\mathcal{S}_0^{\bar{m}}(2-q) }}  \\
&\overset{(b)}{\leq}& 
\begin{cases}
O(1)(\log(1/\sigma))^{\frac{-q+\epsilon}{2}}=o(1) & \mbox{~~if~}\bar{m}>0,  \\
O(1)(\log(1/\sigma))^{\frac{\epsilon+4-3q}{2}}\sigma^{q^2-(2+\epsilon)q+3\epsilon+1}=o(1) & \mbox{~~if~}\bar{m}=0,
\end{cases}
\end{eqnarray*}
where $(a)$ holds even when $\bar{m}=m^*+1$ since $\frac{\epsilon+q-1}{q-1}(2-q)^{\bar{m}}-\frac{\epsilon}{q-1} \leq 0$ by the definition of $m^*$; and $(b)$ is due to the fact that $(\chi^*)^{1/(2-q)} = \Omega \left(\frac{\sigma^{\frac{\epsilon+q-1}{q-1}(2-q)^{\bar{m}-1}-\frac{\epsilon}{q-1}}}{(\log(1/\sigma))^{\mathcal{S}_0^{\bar{m}-1}(2-q)}} \right)$ when $\bar{m}>0$ and $(\chi^*)^{1/(2-q)}=\Omega(\frac{\sigma^{q-\epsilon}}{\log(1/\sigma)})$ when $\bar{m}=0$, according to the choice of $\bar{m}$. For the second integral in \eqref{lq:general:final:rate:two}, note that $(\chi^*)^{\frac{1}{2-q}}=o(|\mathcal{I}_{\bar{m}}|)$, hence $\chi^* |\mathcal{I}_{\bar{m}}|^{q-2}  \rightarrow 0$. It implies that the arguments in calculating the second integral in Case 1 hold here as well. 

So far we have been able to derive the limit of $ \sigma^{q-2} U_2$. We next analyze the term $\sigma^{q-2}\chi^* U_3$ and show that it goes to zero as $\sigma \rightarrow 0$. We have
\begin{eqnarray}
 \lefteqn{ \frac{\chi^*}{\sigma^{2-q}} \int_{0}^\infty
 \int_{\mathcal{K}} \frac{|\eta_q(b/\sigma+z; \chi^*)|^{2-q}}{(|\eta_q(b/\sigma+z; \chi^*)|^{2-q}+ \chi^* q (q-1))^3} \phi(z) dz dF(b) }\nonumber \\
 &\leq& \frac{\chi^*}{\sigma^{2-q}} \int_{0}^\infty
 \int_{\mathcal{K}} \frac{\chi^*\eta_q^{2-q}(1;1)}{( \chi^* q (q-1))^3} \phi(z) dz dF(b)=  \frac{1}{\sigma^{2-q}} \int_{0}^\infty
 \int_{\mathcal{K}} \frac{\eta_q^{2-q}(1;1)}{ \chi^* (q (q-1))^3} \phi(z) dz dF(b) \nonumber 
\end{eqnarray}
The upper bound above has been shown to be zero in the preceding calculations regarding the first term in \eqref{lq:general:dense:final:one}. Furthermore, note that when $z \notin \mathcal{K}$, 
\begin{equation*}
 |\eta_q(b/\sigma +z; \chi^*)|^{2-q} \geq \chi^* \eta^{2-q}_q(1; 1).
\end{equation*}
We can then obtain
\begin{eqnarray*}
\lefteqn{\frac{\chi^*}{\sigma^{2-q}} \int_{0}^\infty
 \int_{\mathcal{I}^{\gamma} \backslash \mathcal{K}} \frac{|\eta_q(b/\sigma+z; \chi^*)|^{2-q}}{(|\eta_q(b/\sigma+z; \chi^*)|^{2-q}+ \chi^* q (q-1))^3} \phi(z) dz dF(b)} \nonumber \\
 &\leq&O(1)\cdot  \frac{1}{\sigma^{2-q}} \int_{0}^\infty
 \int_{\mathcal{I}^{\gamma} \backslash \mathcal{K}} \frac{|\eta_q(b/\sigma+z; \chi^*)|^{4-2q}}{(|\eta_q(b/\sigma+z; \chi^*)|^{2-q}+ \chi^* q (q-1))^3} \phi(z) dz dF(b).
\end{eqnarray*}
The last term has been shown to converge to zero in the analysis of $\sigma^{q-2}U_2$. Above all we have derived that
\begin{eqnarray*}
\lim_{\sigma \rightarrow 0} \frac{\chi^*}{\sigma^{2-q}} \int_{0}^\infty
 \int_{\mathcal{I}^{\gamma}} \frac{|\eta_q(b/\sigma+z; \chi^*)|^{2-q}}{(|\eta_q(b/\sigma+z; \chi^*)|^{2-q}+ \chi^* q (q-1))^3} \phi(z) dz dF(b)=0.
\end{eqnarray*}
This together with the fact 
\[
\frac{\chi^*|\eta_q(b/\sigma+z; \chi^*)|^{2-q}}{\sigma^{2-q}(|\eta_q(b/\sigma+z; \chi^*)|^{2-q}+ \chi^* q (q-1))^3}\leq \frac{q^{-1}(q-1)^{-1}}{|\eta_q(b+\sigma z;\sigma^{2-q}\chi^*)|^{2-q}+\sigma^{2-q}\chi^*q(q-1)},
\]
we can again follow the line of arguments for $-\sigma^{-2}S_2$ in the proof of Lemma \ref{epsilon1:tune:lq:general} to get
\begin{eqnarray*}
\lim_{\sigma \rightarrow 0} \frac{\chi^*}{\sigma^{2-q}} \int_{0}^\infty
 \int_{\mathbb{R} } \frac{|\eta_q(b/\sigma+z; \chi^*)|^{2-q}}{(|\eta_q(b/\sigma+z; \chi^*)|^{2-q}+ \chi^* q (q-1))^3} \phi(z) dz dF(b) = 0.
\end{eqnarray*}
Finally a direct application of Dominated Convergence Theorem gives us $\sigma^{2q-2}U_1\rightarrow q^2\mathbb{E}|B|^{2q-2}$. Hence we are able to derive the following
\begin{eqnarray*}
\lim_{\sigma \rightarrow 0}\frac{\chi^*}{\sigma^q}=\lim_{\sigma \rightarrow 0}\frac{\sigma^{q-2}U_2+\sigma^{q-2}\chi^*U_3}{\sigma^{2q-2}U_1}=\frac{(q-1)\mathbb{E}|B|^{q-2}}{q\mathbb{E}|B|^{2q-2}}.
\end{eqnarray*}
\end{enumerate}
Now that we have derived the convergence rate of $\chi^*$, according to Lemma \ref{epsilon1:tune:lq:general}, we can immediately obtain the order of $R_q(\chi^*,\sigma)$. 
\end{proof}

Having the convergence rate of $R_q(\chi^*_q(\sigma),\sigma)$ as $\sigma \rightarrow 0$ in Lemma \ref{lq:general:final:rate:lemma}, the derivation for the expansion of ${\rm AMSE}(\lambda_{*,q},q,\sigma_w)$ will be the same as the one in the proof of Theorem \ref{thm:densebbiggerthanmu}.

\subsubsection{Proof for the case $q=1$} 
\begin{lemma}\label{l1:amse}
Suppose that $P(|B|\leq t)=\Theta(\sigma^{\ell})$ (as $t\rightarrow 0$) and $\mathbb{E}|B|^2<\infty$, then for $q=1$
\begin{eqnarray*}
&&\alpha_m \sigma^{\ell}  \leq \chi_q^*(\sigma) \leq \beta_m  \sigma^{\ell} (\log_m(1/\sigma))^{\ell/2}, \\
&&\tilde{\alpha}_m \sigma^{2\ell}  \leq 1 - R_q(\chi_q^*(\sigma),\sigma)  \leq \tilde{\beta}_m  \sigma^{2\ell} (\log_m(1/\sigma))^{\ell},
\end{eqnarray*}
for small enough $\sigma$, where $\log_m(1/\sigma)=\underbrace{\log \log \ldots \log}_{m\  \rm times} \left(\frac{1}{\sigma}\right)$; $m> 0$ is an arbitrary integer number; and $\alpha_m, \beta_m, \tilde{\alpha}_m, \tilde{\beta}_m>0$ are four constants depending on $m$.
\end{lemma}
\begin{proof}
Since the proof steps are similar to those in Lemma \ref{dense:l1}, we do not repeat every detail and instead highlight the differences. We write $\chi^*$ for $\chi^*_q(\sigma)$ for notational simplicity. Using the same proof steps in Lemma \ref{dense:l1}, we can obtain $\chi^* \rightarrow 0$, as $\sigma \rightarrow 0$ and 
\begin{eqnarray*}
\chi^*=\frac{\mathbb{E}\phi(\chi^*-B/\sigma)+\mathbb{E}\phi(\chi^*+B/\sigma)}{\mathbb{E}\mathbbm{1}(|Z+B/\sigma|\geq \chi^*)}.
\end{eqnarray*}
Following the same arguments from the proof of Lemma 21 in \cite{weng2016overcoming}, we can show
\begin{eqnarray}
&&\Theta(\sigma^{\ell}) \leq \mathbb{E}\phi(\chi^*-B/\sigma)+\mathbb{E}\phi(\chi^*+B/\sigma)\leq  \Theta(\sigma^{\ell}(\log_m(1/\sigma))^{\ell/2}),  \label{proofneed:one}\\
&& \Theta(\sigma^{\ell}) \leq \mathbb{E}\phi(\sqrt{2}B/\sigma), ~~\mathbb{E}\phi(-B/\sigma+\alpha \chi^*)\leq \Theta(\sigma^{\ell}(\log_m(1/\sigma))^{\ell/2}),  \label{proofneed:two}
\end{eqnarray}
where $\alpha$ is any number between 0 and 1. Since $\mathbb{E}\mathbbm{1}(|Z+B/\sigma|\geq \chi^*) \rightarrow 1$ , the bounds for $\chi^*$ is proved by using the result \eqref{proofneed:one}. Furthermore, we know
\begin{eqnarray*}
&&\hspace{-0.5cm} R_q(\chi^*,\sigma)-1 \leq R_q(\chi,\sigma)-1=\mathbb{E}(\eta_1(B/\sigma+Z;\chi)-B/\sigma-Z)^2+2\mathbb{E}(\partial_1 \eta_1(B/\sigma+Z;\chi)-1)\\
&&\hspace{1.5cm} \leq \chi^2 - 2\mathbb{E}\int_{-B/\sigma-\chi}^{-B/\sigma+\chi}\phi(z)dz= \chi^2 - 4\chi \mathbb{E}\phi(-B/\sigma+\alpha\chi),
\end{eqnarray*}
where $|\alpha|\leq 1$ is dependent on $B$. If we choose $\chi=3e^{-1}\mathbb{E}\phi(\sqrt{2}B/\sigma)$ in the above inequality, it is straightforward to see that 
\[
R_q(\chi^*,\sigma)-1 \leq -\Theta((\mathbb{E}\phi(\sqrt{2}B/\sigma))^2)\leq -\Theta(\sigma^{2\ell}),
\]
where the last step is due to \eqref{proofneed:two}. For the other bound, note that
\begin{eqnarray*}
&&\hspace{-0.5cm}R_q(\chi^*,\sigma)-1=\mathbb{E}(\eta_1(B/\sigma+Z;\chi^*)-B/\sigma-Z)^2+2\mathbb{E}(\partial_1 \eta_1(B/\sigma+Z;\chi^*)-1) \\
&&\geq  -2\mathbb{E} \int_{-B/\sigma-\chi^*}^{-B/\sigma+\chi^*}\phi(z)dz= - 4\chi^* \mathbb{E}\phi(-B/\sigma+\alpha\chi^*) \geq - \Theta(\sigma^{2\ell}(\log_m(1/\sigma))^{\ell}).
\end{eqnarray*}
The last inequality holds because of the upper bound on $\chi^*$ and \eqref{proofneed:two}. 
\end{proof}

 Based on the results of Lemma \ref{l1:amse}, deriving the expansion of ${\rm AMSE(\lambda_{*,1},1,\sigma_w)}$ can be done in a similar way as in the proof of Theorem \ref{thm:densebbiggerthanmu}. We do not repeat it here. 

\subsection{Proof of Theorem \ref{large:thm5}}  \label{proof:large:thm5}

The idea of this proof is similar to those for Theorems \ref{thm:densebbiggerthanmu} and \ref{thm:densegeneralcase}. We make use of the result in Theorem  \ref{thm:eqpseudolip}:
\begin{eqnarray}\label{largedelta:eq}
{\rm AMSE}(\lambda_{*,q},q,\delta)=\bar{\sigma}^2R_q(\chi^*_q(\bar{\sigma}),\bar{\sigma})=\delta \bar{\sigma}^2-\sigma_w^2.
\end{eqnarray}
Since we are in the large sample regime where $\delta \rightarrow \infty$, $\bar{\sigma}$ is a function of $\delta. $ It is clear from \eqref{largedelta:eq} that $0\leq \delta \bar{\sigma}^2-\sigma^2_w\leq \bar{\sigma}^2$. Hence $\bar{\sigma}^2\leq \sigma^2_w/(\delta-1) \rightarrow 0$, which further leads to
\begin{eqnarray}\label{useful:rate}
\bar{\sigma}^2=\frac{\sigma_w^2}{\delta}+o(1/\delta).
\end{eqnarray}
Due to the fact that $\bar{\sigma}\rightarrow 0$ as $\delta \rightarrow \infty$, we will be able to use the convergence rate results of $R_q(\chi^*_q(\sigma),\sigma)$ (as $\sigma \rightarrow 0$) we have proved in Lemmas \ref{lq12:amse} and \ref{dense:l1}. For $1<q\leq 2$, Equations \eqref{largedelta:eq}, \eqref{useful:rate} and Lemma \ref{lq12:amse} together yield
\begin{eqnarray}\label{forl1:purpose}
&& \delta^2 ({\rm AMSE}(\lambda_{*,q},q,\delta)-\sigma_w^2/\delta)=\delta^2(\bar{\sigma}^2R_q(\chi^*_q(\bar{\sigma}),\bar{\sigma})-(\bar{\sigma}^2-\bar{\sigma}^2R_q(\chi^*_q(\bar{\sigma}),\bar{\sigma})/\delta)) \nonumber \\
&&=(\bar{\sigma}^4\delta^2)\cdot \frac{R_q(\chi^*_q(\bar{\sigma}),\bar{\sigma})-1}{\bar{\sigma}^2}+(\delta \bar{\sigma}^2) \cdot R_q(\chi^*_q(\bar{\sigma}),\bar{\sigma}) \\
&&\rightarrow \frac{-(q-1)^2(\mathbb{E}|B|^{q-2})^2}{\mathbb{E}|B|^{2q-2}}\sigma_w^4+\sigma_w^2. \nonumber
\end{eqnarray}
For the case $q=1$, from Lemma \ref{dense:l1} we know $R_q(\chi^*_q(\bar{\sigma}),\bar{\sigma})-1$ is exponentially small. So the firs term in \eqref{forl1:purpose} vanishes and the second term remains the same.

\subsection{Proof of Theorem \ref{large:thm6}} \label{proof:large:thm6}

Theorem \ref{large:thm6} can be proved in a similar fashion as for Theorem \ref{large:thm5}. Equation \eqref{useful:rate} still holds. Equations \eqref{largedelta:eq}, \eqref{useful:rate} and Lemma \ref{l1:amse} together give us for $q=1$,
\begin{eqnarray*}
 \delta^{\ell+1} ({\rm AMSE}(\lambda_{*,q},q,\delta)-\sigma_w^2/\delta)=(\bar{\sigma}^{2\ell+2}\delta^{\ell+1})\cdot \frac{R_q(\chi^*_q(\bar{\sigma}),\bar{\sigma})-1}{\bar{\sigma}^{2\ell}}+(\delta^{\ell} \bar{\sigma}^2) \cdot R_q(\chi^*_q(\bar{\sigma}),\bar{\sigma}),
\end{eqnarray*}
where the first term above is $\Theta(1)$ and the second one is $o(1)$ when $\ell<1$. The case $1<q\leq 2$ can be proved exactly the same way as in Theorem \ref{large:thm5} by using Lemma \ref{lq:general:final:rate:lemma}.

\section*{Acknowledgment}

Arian Maleki is supported by NSF grant CCF1420328.

\bibliography{mc}
\bibliographystyle{ims}

\end{document}